\documentclass[10pt]{amsart}
\usepackage[bb=fourier,cal=euler]{mathalfa}
\usepackage{amssymb}
\usepackage{mathtools}		
\usepackage{mathabx}		  
\usepackage{nicefrac}
\usepackage{indentfirst}	
\usepackage{amsthm}
\usepackage{thmtools}
\usepackage{enumitem}		
\usepackage[backref=page,colorlinks=true]{hyperref}	
\usepackage[capitalise]{cleveref} 

\usepackage[font=footnotesize]{caption}

\usepackage{xcolor}
\hypersetup{
    colorlinks,
    linkcolor={red!50!black},
    citecolor={blue!50!black},
    urlcolor={blue!80!black}
}

\renewcommand*{\backref}[1]{}
\renewcommand*{\backrefalt}[4]{\quad \tiny 
  \ifcase #1 (\textbf{NOT CITED.})%
  \or    (Cited on page~#2.)%
  \else   (Cited on pages~#2.)%
  \fi}

\makeatletter

\def\MRbibitem{\@ifnextchar[\my@lbibitem\my@bibitem}

\def\mybiblabel#1#2{\@biblabel{{\hyperref{http://www.ams.org/mathscinet-getitem?mr=#1}{}{}{#2}}}}

\def\myhyperanchor#1{\Hy@raisedlink{\hyper@anchorstart{cite.#1}\hyper@anchorend}}

\def\my@lbibitem[#1]#2#3#4\par{%
  \item[\mybiblabel{#2}{#1}\myhyperanchor{#3}\hfill]#4%
  \@ifundefined{ifbackrefparscan}{}{\BR@backref{#3}}%
  \if@filesw{\let\protect\noexpand\immediate
    \write\@auxout{\string\bibcite{#3}{#1}}}\fi\ignorespaces%
}

\def\my@bibitem#1#2#3\par{%
  \refstepcounter\@listctr
  \item[\mybiblabel{#1}{\the\value\@listctr}\myhyperanchor{#2}\hfill]#3%
  \@ifundefined{ifbackrefparscan}{}{\BR@backref{#2}}%
  \if@filesw\immediate\write\@auxout
    {\string\bibcite{#2}{\the\value\@listctr}}\fi\ignorespaces%
}

\makeatother

\declaretheoremstyle[
headfont=\small\itshape,
bodyfont=\small
]{myremark}

\declaretheorem[numberwithin=section]{theorem}
\declaretheorem[sibling=theorem]{lemma}

\declaretheorem[sibling=theorem]{proposition}

\declaretheorem[sibling=theorem,style=remark]{remark}
\declaretheorem[sibling=theorem,style=remark]{question}
\declaretheorem[name=Acknowledgements, style=remark, numbered=no]{ack}

\hypersetup{bookmarksdepth = 3} 
\numberwithin{equation}{section}     

\crefname{subsection}{Subsection}{Subsections}
\Crefname{subsection}{Subsection}{Subsections}

\setlist[enumerate,1]{label={\upshape(\alph*)},ref=\alph*}
\setlist[enumerate,2]{label={\upshape(\arabic*)},ref=\arabic*}

\crefname{enumi}{part}{parts}

\newcommand{\C}{\mathbb{C}}
\newcommand{\R}{\mathbb{R}}
\newcommand{\Z}{\mathbb{Z}}

\newcommand{\T}{\mathbb{T}}
\renewcommand{\P}{\mathbb{P}}

\newcommand{\SL}{\mathrm{SL}}

\newcommand{\cC}{\mathcal{C}}

\newcommand{\cH}{\mathcal{H}}
\newcommand{\cK}{\mathcal{K}}\newcommand{\cL}{\mathcal{L}}

\newcommand{\cS}{\mathcal{S}}

\newcommand{\fJ}{\mathfrak{J}}\newcommand{\fL}{\mathfrak{L}}
\newcommand{\fM}{\mathfrak{M}}

\newcommand{\st}{\;\mathord{;}\;}

\newcommand{\dd}{\,\mathrm{d}}   
\newcommand{\area}[1]{\lvert#1\rvert}
\newcommand{\vol}[1]{\lvert#1\rvert}

\newcommand{\arxiv}[1]{Preprint \href{http://arxiv.org/abs/#1}{arXiv:{#1}}}
\newcommand{\doi}[1]{\href{http://dx.doi.org/#1}{doi:{#1}}}

\renewcommand{\epsilon}{\varepsilon}
\renewcommand{\setminus}{\smallsetminus}

\begin{document}

\title[On the approximation of convex bodies by ellipses]{On the approximation of convex bodies by ellipses with respect to the symmetric difference metric}

\author{Jairo Bochi}
\email{\href{mailto:jairo.bochi@mat.uc.cl}{jairo.bochi@mat.uc.cl}}
\address{Facultad de Matem\'aticas, Pontificia Universidad Cat\'olica de Chile. Avenida Vicu\~na Mackenna 4860, Santiago, Chile}
\thanks{Partially supported by projects Fondecyt 1180371 and Conicyt PIA ACT172001.}

\date{June 3, 2018}




\maketitle

\begin{abstract}
Given a centrally symmetric convex body $K \subset \R^d$ and a positive number $\lambda$, we consider, among all ellipsoids $E \subset \R^d$ of volume $\lambda$, those that best approximate $K$ with respect to the symmetric difference metric, or equivalently that maximize the volume of $E\cap K$: these are the maximal intersection (MI) ellipsoids introduced by Artstein-Avidan and Katzin. 
The question of uniqueness of MI ellipsoids (under the obviously necessary assumption that $\lambda$ is between the volumes of the John and the Loewner ellipsoids of $K$) is open in general.
We provide a positive answer to this question in dimension $d=2$. 
Therefore we obtain a continuous $1$-parameter family of ellipses interpolating between the John and the Loewner ellipses of $K$.
In order to prove uniqueness, we show that the area $I_K(E)$ of the intersection $K \cap E$ is a strictly quasiconcave function of the ellipse $E$, with respect to the natural affine structure on the set of ellipses of area $\lambda$.  
The proof relies on smoothening $K$, putting it in general position, and obtaining uniform estimates for certain derivatives of the function $I_K(\mathord{\cdot})$.
Finally, we provide a characterization of maximal intersection positions, that is, the situation where the MI ellipse of $K$ is the unit disk, under the assumption that the two boundaries are transverse.
\end{abstract}

\section{Introduction}

\subsection{Convex bodies and approximation problems}

The euclidian distance induces the well-known \emph{Hausdorff metric} on the set $\cS^d$ of nonempty compact subsets of $\R^d$. Namely, $d_\mathrm{Haus}(K_1,K_2)$ is defined as the least $\epsilon\ge 0$ such that every point in one of the sets $K_i$ is within euclidian distance at most $\epsilon$ from some point in the other set. 
By Blaschke selection theorem \cite[p.~37]{Falconer}, bounded subsets of $\cS^d$ are compact; in particular the metric space $(\cS^d, d_\mathrm{Haus})$ is complete and locally compact.

We are interested in the space $\cK^d$ of \emph{convex bodies} (i.e., compact convex sets with nonempty interior), which is a locally closed subset of $\cS^d$.
There are other natural metrics on $\cK^d$ that also induce the Hausdorff topology: see \cite{SW}.
Among these, we highlight the \emph{symmetric difference metric} and the \emph{normalized symmetric difference metric}: 
\begin{equation}\label{e.metrics}
d_\mathrm{sym}  (K_1, K_2) \coloneqq \vol{K_1 \vartriangle K_2} \, , \qquad
d_\mathrm{nsym} (K_1, K_2) \coloneqq \frac{\vol{K_1 \vartriangle K_2}}{\vol{K_1 \cup K_2}}  \, 
\end{equation}
where $\vol{\mathord{\cdot}}$ denotes volume (Lebesgue measure) in $\R^d$.
These two metrics make sense in broader classes of sets and are known in Measure Theory as the Fr\'echet--Nikodym and the Marczewski--Steinhaus pseudometrics, respectively.
Note that the metric $d_\mathrm{Haus}$ is preserved by the action euclidian isometries of $\R^d$, while $d_\mathrm{sym}$ is preserved by volume-preserving affine transformations, and $d_\mathrm{nsym}$ is preserved by all affine transformations.
In this paper we focus on the symmetric difference metric $d_\mathrm{sym}$. 

There is a large body of literature on approximation of convex bodies by simpler ones, as e.g.\ polyhedra: see the survey articles \cite{Gruber93,Bron}. Let us mention a few of the most classical results. 
Given a plane convex body $K \in \cK^2$, for each $n\ge 3$, let $P_n^{(1)}$ be an inscribed $n$-gon of maximal area, let $P_n^{(3)}$ be a circumscribed $n$-gon of minimal area, and let $P_n^{(2)}$ be a convex $n$-gon that best approximates $K$ with respect to the symmetric difference metric. The approximation errors $\epsilon_n^{(i)} \coloneqq d_\mathrm{sym} \big(K,P_n^{(i)}\big)$ obviously tend to zero. Dowker \cite{Dowker} (see also \cite[\S~II.3]{FT}) proved that the sequence $(\epsilon_n^{(1)})$ is concave and the sequence $(\epsilon_n^{(3)})$ is convex, and Eggleston \cite{Egg} proved that the sequence $(\epsilon_n^{(2)})$ is also convex. On the other hand, L.~Fejes T\'oth stated in his famous book \cite[p.~43]{FT} that if $\partial K$ is sufficiently differentiable and positively curved then each of these three sequences is asymptotic to $c_i n^{-2}$, for some explicitly defined constant $c_i = c_i(K)>0$; curiously, $(c_1,c_2,c_3)$ is proportional to $\big( 1, \tfrac{3}{4}, 2 \big)$.  These formulas were later proved by McClure and Vitale \cite{MCV} for $i=1$ and $3$, and by Ludwig \cite{Ludwig} for $i=2$. For higher-dimensional versions of these results, see \cite{GK,Gruber93,Ludwig}.

Another class of ``simple'' convex bodies consists on ellipsoids. 
Let us note that ellipsoids are the convex bodies that are worst approximable by polytopes: see \cite[Rem.~2]{Ludwig}.

It is well-known that every convex body $K$ admits a unique inscribed ellipsoid $\fJ_K$ of maximal volume and a unique circumscribed ellipsoid  $\fL_K$ of minimal volume; they are called respectively the \emph{John ellipsoid} and the \emph{Loewner ellipsoid} of $K$. 
Moreover, if $K$ is \emph{centrally symmetric} in the sense that $K=-K$, then so are the ellipsoids $\fJ_K$ and $\cL_K$.
See \cite[Lecture~3]{Ball97} for proofs, \cite{Henk} for historical information, and \cite[\S~10.12]{Schneider} for other types of ellipsoids associated to a convex body.

Our original motivation comes from the following approximation problem posed by W.~Kuperberg \cite{Kuperberg}:

\begin{question}\label{q.Wlodek}
If $K$ is a plane convex body of area~$1$, and if $E$ is an ellipse of area~$1$ that minimizes $d_\mathrm{sym}(K,E)$ among all such ellipses, is $E$ necessarily unique?
\end{question}

In this paper, we answer this question positively under the assumption that $K$ is centrally symmetric. 
Actually, we prove uniqueness of a family of a certain ellipses that includes $E$ and interpolates between the John and the Loewner ellipses, as explained below.

\subsection{Maximal intersection ellipsoids}

Let $\cC^d \subset \cK^d$ denote the set of centrally symmetric $d$-dimensional convex bodies, where $d \ge 2$.
Following Artstein-Avidan and Katzin \cite{AAK}, we say that an ellipsoid $E \subset \R^d$ is a \emph{maximal intersection (MI) ellipsoid} for $K \in \cC^d$ if among all ellipsoids with the same volume as $E$, it maximizes the volume of $E \cap K$.
In view of the relation
\begin{equation}\label{e.inclusion_exclusion}
d_\mathrm{sym}(K,E) =  \vol{K} + \vol{E} - 2 \vol{K \cap E} \, ,
\end{equation}
it is equivalent to say that $E$ is an optimal approximation for $K$ with respect to the symmetric difference metric, among all ellipsoids of a fixed volume. 

Immediate examples of MI ellipsoids are the John ellipsoid $\fJ_K$ and the Loewner ellipsoid $\fL_K$. Furthermore, there are no other MI ellipsoids with volume $\vol{\fJ_K}$ or $\vol{\fL_K}$.
On other hand, if $\lambda>0$ is either smaller than $\vol{\fJ_K}$ or bigger than $\vol{\fL_K}$ then $K$ obviously admits infinitely many MI ellipsoids of volume $\lambda$. 
Artstein-Avidan and Katzin \cite{AAK} ask whether uniqueness of MI ellipsoids holds when $\lambda$ is in the interesting range $\vol{\fJ_K} < \lambda < \vol{\fL_K}$. We provide a positive answer for this question in dimension $d=2$:

\begin{theorem}\label{t.unique}
Let $K \subset \R^2$ be a centrally symmetric convex body, and let $\lambda$ be a number in the range $\area{\fJ_K} \le \lambda \le \area{\fL_K}$. Then there exists a unique MI ellipse $\fM_K(\lambda)$ of area $\lambda$, and it is centrally symmetric. 
\end{theorem}

In particular, taking $\lambda = \area{K}$, we obtain the announced positive answer for \cref{q.Wlodek} in the centrally symmetric case.

As a simple consequence of uniqueness (using the Blaschke selection theorem), the MI ellipse $\fM_K(\lambda)$ provided by \cref{t.unique} depends continuously on both $K$ and $\lambda$, provided that $\area{\fJ_K} \le \lambda \le \area{\fL_K}$.
In particular, these MI ellipses continuously interpolate between the John and Loewner ellipses.

As remarked in \cite{AAK}, every centrally symmetric convex body in $\R^d$ admits MI ellipsoids of any prescribed volume $\lambda>0$ that are centrally symmetric. In dimension $2$, as an ingredient of the proof of \cref{t.unique}, we need to establish the following:

\begin{lemma}\label{l.discard}
Let $K \subset \R^2$ be a centrally symmetric convex body, and let $\lambda$ be a number in the range $\area{\fJ_K} <  \lambda < \area{\fL_K}$. Then every MI ellipse of area $\lambda$ for $K$ is centrally symmetric.
\end{lemma}

\subsection{Quasiconcavity of the area function}

\cref{t.unique} follows from a sharper result.
In order to state it, let us introduce some notation.

Given $\lambda>0$, let $\cC^2_\lambda$ be the set of centrally symmetric bodies $K \in \cC^2$ that satisfy
$\area{\fJ_K} < \lambda < \area{\fL_K}$. 
Note that $\cC^2_\lambda$ is an open subset of $\cC^2$.
Consider the family of ellipses of area $\pi$ whose axes are the $x$ and $y$ axes (together with the unit disk), which we parameterize  by $t \in \R$ as follows:
\begin{equation}\label{e.standard_family}
E_t \coloneqq \big\{ (x,y) \in \R^2 \st e^{t} x^2 + e^{-t}y^2 \le 1 \big\}.
\end{equation}
For any $K \in \cC^2$, its \emph{intersection function} is the function $I_K \colon \R \to \R$ defined by:
\begin{equation}\label{e.I_def}
I_K (t) \coloneqq \area{E_t \cap K} \, .
\end{equation}

Recall that a real function $f$ defined on an interval $J \subseteq \R$ is called \emph{quasiconcave} if for every $s$, $t$, $u \in J$,
$$
t_0 < t_1 < t_2 \ \Rightarrow \ f(t_1) \ge \min\{ f(t_0), f(t_2)\} \, ,
$$
and is called \emph{strictly quasiconcave} if the inequality on the right is always strict.

Our crucial technical result, whose proof occupies the bulk of this paper, is the following:

\begin{theorem}\label{t.qc}
For every $K \in \cC^2_\pi$, the associated intersection function $I_K$ is strictly quasiconcave.
\end{theorem}

If \cref{l.discard} and \cref{t.qc} are assumed, we can immediately deduce our uniqueness result:

\begin{proof}[Proof of \cref{t.unique}]
Let $K \in \cC^2$.
Since the John and Loewner ellipses are known to be unique and centrally symmetric, it is sufficient to consider $\area{\fJ_K} <  \lambda < \area{\fL_K}$. Applying an homothecy if necessary, we can assume that $\lambda = \pi$, so $K \in \cC^2_\pi$.
As remarked before, $K$ admits at least one MI ellipse of area $\pi$.
Suppose there are two, say $E \neq E'$.
By \cref{l.discard}, these ellipses must be centrally symmetric.
Applying an appropriate element of $\SL(2,\R)$ (i.e., a linear map of determinant $1$), we can assume that $E$ and $E'$ are elements of the family \eqref{e.standard_family}.
So the intersection function $I_K$ attains its maximum at two distinct points, contradicting \cref{t.qc}. 
This proves uniqueness of the MI ellipse of area $\pi$.
\end{proof}

\begin{remark}
The space $\cH$ of centrally symmetric ellipses of area $\pi$ has a natural affine structure, which is in fact equivalent to the affine structure of the hyperbolic plane: see \cite{footprint}. As a reformulation of \cref{t.qc}, for every $K \in \cC^2_\pi$, the function $E \in \cH \mapsto \area{E \cap K}$ is strictly quasiconcave with respect to this affine structure.  
\end{remark}

In fact we will prove a more general version of \cref{t.qc}: see \cref{t.extension} below.

\subsection{Maximum intersection position}

Let us say that a centrally symmetric convex body $K \subset \R^d$ is in \emph{maximum intersection (MI) position} if the euclidian unit ball in $\R^d$ is a MI ellipsoid for $K$.
Every centrally symmetric convex body can be put in MI position by applying an appropriate invertible linear map, whose determinant can be any prescribed non-zero number.

In \cref{s.MI_position} we will give a simple characterization of MI position in the plane under a transversality hypothesis, which has the following interesting consequence:

\begin{proposition}\label{p.48}
Let $K \subset \R^2$ is a compact convex centrally symmetric set whose boundary $\partial K$ is transverse to the unit circle $S^1$. 
Then:
\begin{enumerate}
\item\label{i.4} 
if the intersection consists of $4$ points then $K$ cannot be in MI position;
\item\label{i.8} 
if the intersection consists of $8$ points then $K$ is in MI position if and only if $\partial K \cap S^1$ is invariant by a quarter turn (i.e., rotation of $\pi/2$).
\end{enumerate}
\end{proposition}

\cref{f.quarterturn} shows an example of situation (\ref{i.8}).

\begin{figure}[hbt]
	\begin{center}
		\includegraphics[width=.65\textwidth]{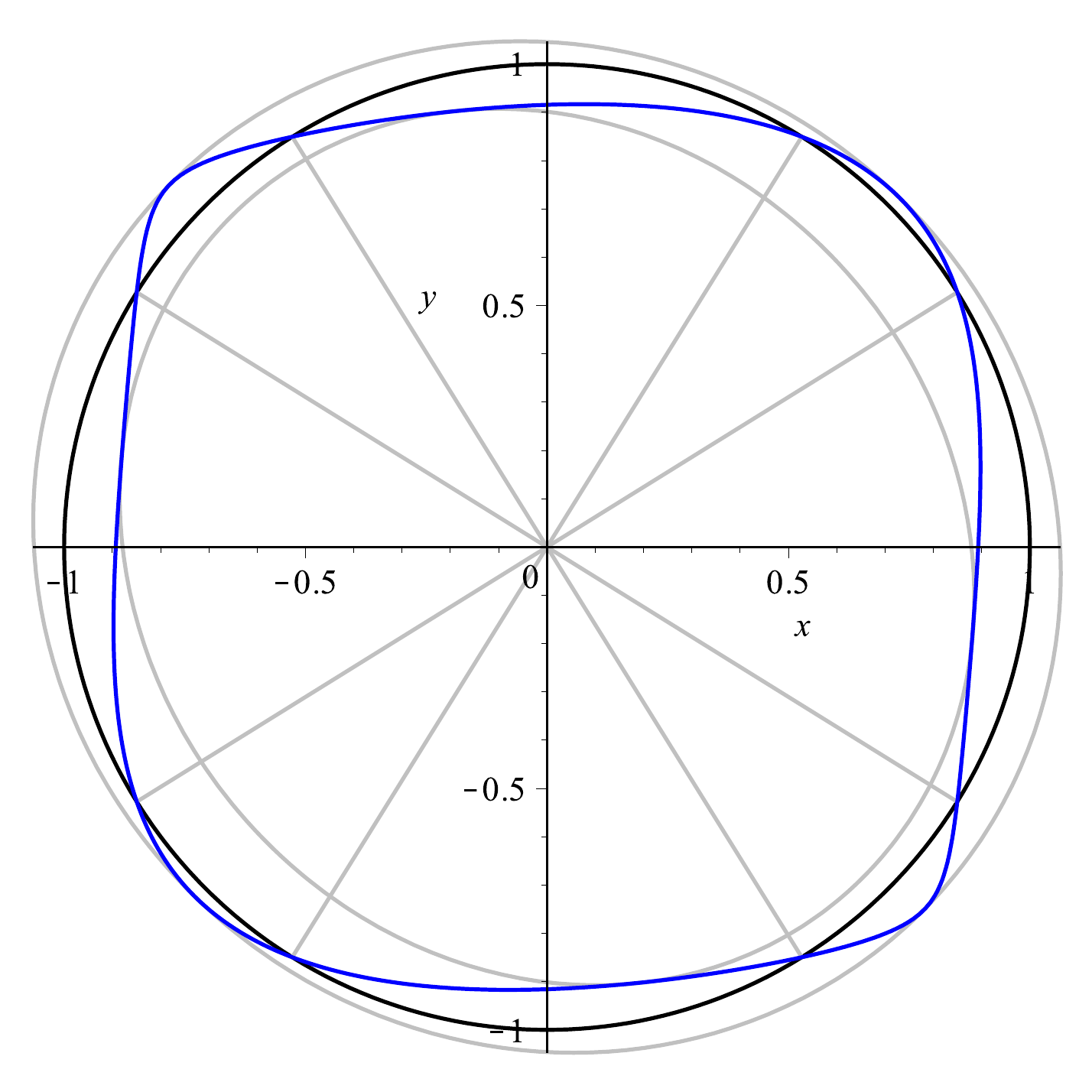}
	\end{center}
	\caption{Example of a centrally symmetric body $K$ in MI position; the unit circle and the John and Loewner ellipses of $K$ are also pictured. The curve $\partial K$ satisfies the equation $1.355 x^2-0.58xy+1.005 y^2 -0.1264 x^4 + 0.58 x^3 y - 1.041 x^2 y^2 + 0.58 x y^3+ 0.2236 y^4=1$. Apart from being centrally symmetric, $K$ has no other linear symmetries.}
	\label{f.quarterturn}
\end{figure}

In \cref{s.MI_position} we also discuss the classical characterization of the John position (that is, the situation when the John ellipsoid is round) and how it relates to MI position.

\subsection{Strategy of the proofs and organization of the paper}

The proof of \cref{t.qc} occupies \cref{s.derivatives,s.implication}. In order to make this proof more digestible, let us highlight the key ideas. 
We perform a local study of the function $I_K(t) = \area{K \cap E_t}$. It is essentially sufficient to consider a neighborhood of $t=0$. 
We initially assume that the curve $\partial K$ is smooth and crosses the unit circle $S^1 = \partial E_0$ at finitely many points, making nonzero angles. This transversality condition implies that the function $I_K$ is of class $C^2$ on a neighborhood of $0$; furthermore, there are explicit formulas for the first two derivatives of $I_K$ at $0$, with $I_K'(0)$ depending on the locations of the crossings between $\partial K$ and $S^1$, and $I_K''(0)$ depending also on the crossing angles: see \cref{p.formulas}.
Another observation (\cref{p.tangency}) is that we can allow certain types of ``tame'' tangencies between $\partial K$ and $S^1$ and the function $I_K(t)$ will still be $C^1$ on a neighborhood of $t=0$, though $I''(0)$ may fail to exist.
Then we reach the heart of the whole proof, \cref{p.key}, which essentially says that if $\partial K$ is transverse to $S^1$ and $I_K'(0) = 0$ then $I_K$ is strictly concave around $0$, that is, $I_K''(0)<0$. 
The proof of this key \lcnamecref{p.key} relies on a lower bound \eqref{e.gatito} for $-I_K''(0)$ which, like the formula for $I_K'(0)$, depends only on the locations and not on the angles of the crossings between $\partial K$ and $S^1$. 
A quick inspection of this bound reveals that it has a strong tendency to be positive: for example, if each pair of consecutive crossings is separated by a circle arc of length $< \pi/2$ then the bound is automatically positive. 
The actual proof of \cref{p.key} is done by a case-by-case analysis, which occupies \cref{ss.key_proof}.
All estimates are explicit and ultimately we obtain a positive lower bound for $\max\{ |I_K'(0)|, - I_K''(0) \}$ that does not depend on $K$, but only on the areas of $K$ and $K \cap E_0$. This uniformity with respect to $K$ is crucial for the second part of the proof of \cref{t.qc}, presented in \cref{s.implication}. 
There, we argue that any convex body $K \in \cC^2_\pi$ admits small perturbations $\tilde K$ with respect to the symmetric difference metric that have the same area as $K$ and are ``regular'' in the following sense: the boundary $\partial \tilde K$ is smooth and transverse to all ellipses $\partial E_t$, except for a finite number of ``tame'' tangencies. Using the estimates obtained previously, we conclude that the resulting function $I_{\tilde K}$ is strictly quasiconcave in a quantitative sense that is independent of the size of the perturbation. This uniformity allows us to take a limit and conclude that $I_K$ is strictly quasiconcave as well, therefore proving \cref{t.qc}.

The paper has three additional short sections.
In \cref{s.discard} we prove \cref{l.discard} and therefore conclude the proof of \cref{t.unique}.
In \cref{s.MI_position} we study MI positions. 
These two sections may be read independently from the previous ones, except that we use 
\cref{p.formulas} and \cref{t.extension}.
Finally, \cref{s.questions} discusses possible extensions of our results.

\section{Derivatives of the intersection function under regularity assumptions}
\label{s.derivatives}

\subsection{Differentiability of the intersection function}

Consider a pair of Jordan curves $\Gamma_1$, $\Gamma_2$ in the plane. 
A point $p$ of intersection between the curves is called:
\begin{itemize} 
	\item a \emph{crossing} if each curve admits a $C^1$ parameterization at a neighborhood of $p$, and the pair of tangent vectors at $p$ is linearly independent;
	\item a \emph{quadratic tangency} if each curve admits a $C^2$ parameterization at a neighborhood of $p$, and these parametrized curves have a first- but not a second-order contact at $p$.
\end{itemize}
We say that the  curves $\Gamma_1$, $\Gamma_2$ are:
\begin{itemize}
	\item \emph{transverse} if every point of intersection is a crossing;
	\item \emph{quasitransverse} if every point of intersection is either a crossing or a quadratic tangency.
\end{itemize}
In either case, the number of intersections is finite.

\medskip

Now consider a centrally symmetric convex body $K \subset \R^2$ whose boundary $\partial K$ is transverse to the unit circle $\partial E_0 = S^1$. Then the two curves cross at $4n$ points. If $K \in \cC^2_\pi$ then necessarily $n\ge 1$.
We list the crossing points in counterclockwise order as $\zeta_1$, \dots, $\zeta_{4n}$. 
Since $K$ is centrally symmetric, we have $\zeta_{j+2n} = - \zeta_j$.
Shifting indices by $1$ (mod~$4n$) if necessary, we assume that the following condition holds: if the curve $\partial K$ is traversed counterclockwise, then it exits the unit disk $E_0$ at the points $\zeta_j$ with $j$ even,
and enters it at the points $\zeta_j$ with $j$ odd: see \cref{f.notation}.
Let $\alpha_j > 0$ denote the non-oriented angles of intersection; note that $\alpha_j < \pi/2$ since $K$ is centrally symmetric.
Fix numbers $\xi_1 < \xi_2 < \dots < \xi_{4n} < \xi_1+2\pi$ such that $\zeta_j = (\cos \xi_j, \sin \xi_j)$.

\begin{center}
	\begin{figure}[hbt]
	\includegraphics[width=.65\textwidth]{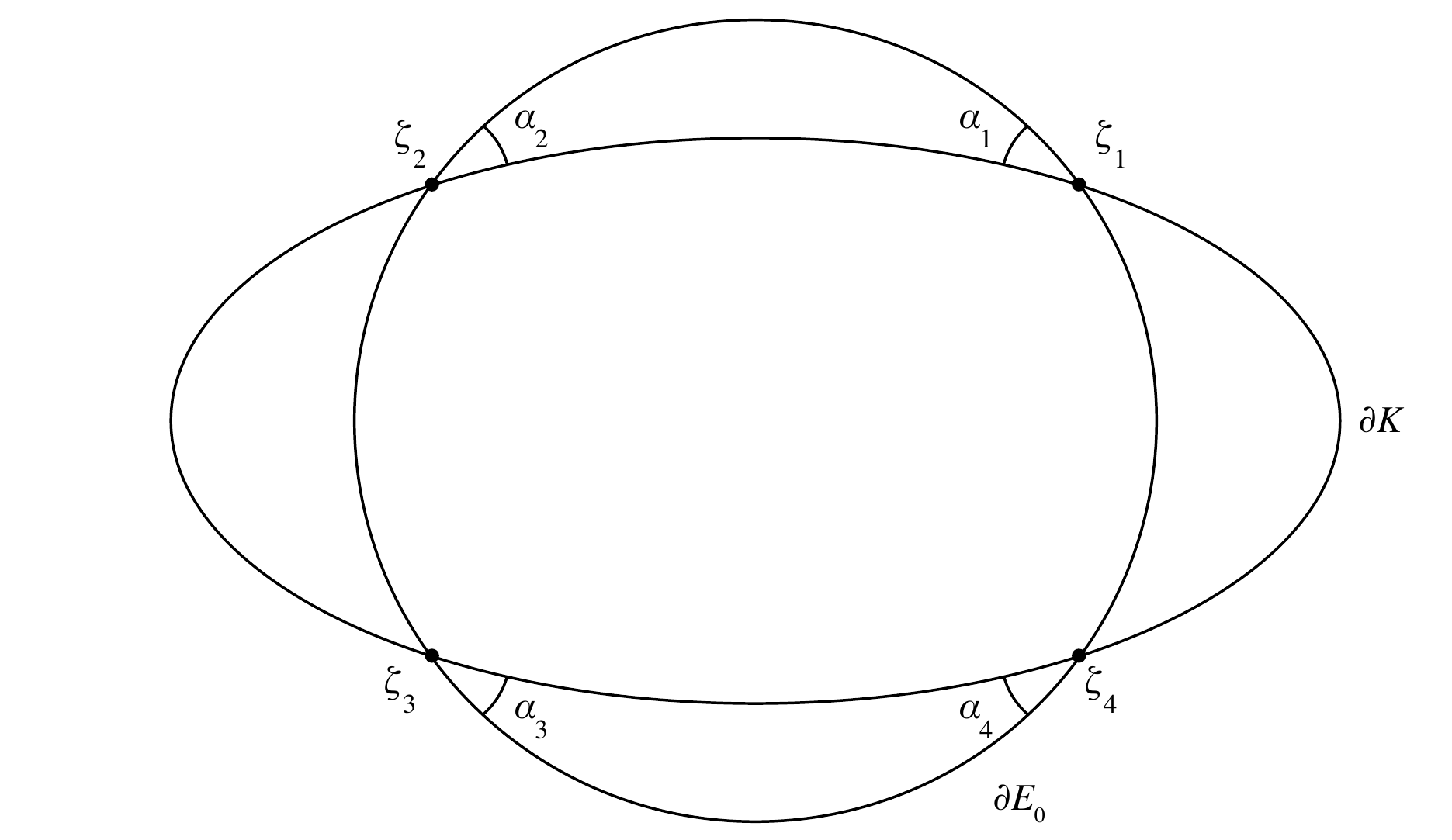}
	\caption{Intersections between the curve $\partial K$ and the circle $\partial E_0 = S^1$.}\label{f.notation}
	\end{figure}
\end{center}

\begin{proposition}\label{p.formulas}
Let $K \subset \R^2$ be a centrally symmetric convex body whose boundary $\partial K$ is transverse to the unit circle $\partial E_0$ and intersects it at $4n > 0$ points. 
Let $(\xi_j)$ and $(\alpha_j)$ be the crossing positions and angles as defined above.
Then the intersection function $I=I_K$ is $C^2$ at a neighborhood of $0$ and
\begin{align}
  I'(0) &= \frac{1}{2} \sum_{j=1}^{2n} (-1)^j \sin 2\xi_j \, , \label{e.1st} \\
-I''(0) &= \frac{1}{4} \sum_{j=1}^{2n} \left[ (-1)^j \sin 4\xi_j
+ \frac{1 + \cos 4\xi_j}{\tan \alpha_j} \right] \, . \label{e.2nd}
\end{align}
\end{proposition}

\begin{proof} 
In polar coordinates $(r,\theta)$, the ellipse $\partial E_t$ has equation:
$$
r^2 = \frac{1}{e^t \cos^2 \theta + e^{-t} \sin^2 \theta} \, .
$$
Similarly, the curve $\partial K$ is represented by some equation $r^2 = G(\theta)$, where $G$ is a positive function on the circle $\R/2\pi\Z$ which satisfies $G(\theta+\pi) = G(\theta)$.
Furthermore, $G-1$ vanishes exactly on the points $\xi_1$, \dots, $\xi_{4n}$, is $C^1$ on a neighborhood of these points, and
\begin{equation}\label{e.signs}
(-1)^j G'(\xi_j) > 0
\end{equation}
for each $j$.

Let $f(\theta,t) \coloneqq G(\theta) - 1/(e^t \cos^2 \theta + e^{-t} \sin^2 \theta)$.
By the Implicit Function Theorem, for $t$ sufficiently close to $0$, the function $f(\mathord{\cdot},t)$ vanishes on $4n$ points $\xi_1(t)$, \dots, $\xi_{4n}(t)$; moreover each function $\xi_j(\mathord{\cdot})$ is $C^1$ and satisfies $\xi_j(0) = \xi_j$ and 
\begin{equation*}
\xi_j'(t) = - \frac{f_t(\xi_j(t),t)}{f_\theta(\xi_j(t),t)} \, .
\end{equation*}
Consider the function:
$$
A(t) \coloneqq d_\mathrm{sym}(E_t,K) 
= \frac{1}{2}\int_0^{2\pi} |f(\theta,t)| \dd \theta =  \int_0^{\pi} |f(\theta,t)| \dd \theta \, .
$$
Equivalently,
\begin{equation}\label{e.L1} 
A(t) = \sum_{j=1}^{2n} (-1)^j \int_{\xi_j(t)}^{\xi_{j+1}(t)} f(\theta,t) \dd \theta  \, . 
\end{equation}
By Leibniz integral rule, 
\begin{align}
A'(t) 
&= \sum_{j=1}^{2n} (-1)^j \Biggl[ \int_{\xi_j(t)}^{\xi_{j+1}(t)} f_t(\theta,t) \dd \theta + \underbrace{f(\xi_{j+1}(t),t)}_0 \xi_{j+1}'(t) -  \underbrace{f(\xi_j(t),t)}_{0} \xi_j'(t)\Biggr] \, ,
\label{e.Leibniz1}\\
A''(t) 
&= \sum_{j=1}^{2n} (-1)^j \left[ \int_{\xi_j(t)}^{\xi_{j+1}(t)} f_{tt}(\theta,t) \dd \theta + f_t(\xi_{j+1}(t),t) \xi_{j+1}'(t) -  f_t(\xi_j(t),t) \xi_j'(t) \right] \notag \\
&= \sum_{j=1}^{2n} (-1)^j \left[  \int_{\xi_j(t)}^{\xi_{j+1}(t)} f_{tt}(\theta,t) \dd \theta- 2 f_t(\xi_j(t),t) \xi_j'(t) \right] \notag \\
&= \sum_{j=1}^{2n} (-1)^j \left[  \int_{\xi_j(t)}^{\xi_{j+1}(t)} f_{tt}(\theta,t) \dd \theta + \frac{2 [f_t(\xi_j(t),t)]^2}{f_\theta(\xi_j(t),t)} \right] \, . \label{e.Leibniz2}
\end{align}
In particular, $A$ is a $C^2$ function on a neighborhood of $0$.
Since the functions $A$ and $I$ are related by formula~\eqref{e.inclusion_exclusion}, $I$ is also $C^2$ on a neighborhood of $0$.

Now we consider $t=0$.
A computation gives:
$$
f_t(\theta,0)    = \cos 2 \theta \quad \text{and} \quad 
f_{tt}(\theta,0) = -\cos 4\theta.
$$
Plugging into \eqref{e.Leibniz1},
$$
I'(0)
= - \frac{1}{2} A'(0) 
= - \frac{1}{4} \sum_{j=1}^{2n} (-1)^j \big[ \sin 2 \xi_{j+1} - \sin 2 \xi_j \big]
= \frac{1}{2} \sum_{j=1}^{2n} (-1)^j \sin 2 \xi_j \, ,
$$
proving \eqref{e.1st}.
Analogously, from \eqref{e.Leibniz1} we obtain:
\begin{align*}
-I''(0) = \frac{1}{2} A''(0) 
&= \frac{1}{2} \sum_{j=1}^{2n} (-1)^j \left[ \frac{- \sin 4 \xi_{j+1} + \sin 4 \xi_j}{4} + \frac{2 \cos^2 2 \xi_j}{G'(\xi_j)} \right] \\
&= \frac{1}{2} \sum_{j=1}^{2n} \left[ \frac{(-1)^j  \sin 4 \xi_j}{2} + \frac{1 + \cos 4\xi_j}{|G'(\xi_j)|} \right] \, ,
\end{align*}
where in the last step we have used \eqref{e.signs}.
Since $|G'(\xi_j)| = 2 \tan \alpha_j$, we obtain formula \eqref{e.2nd}.
\end{proof}

\begin{proposition}\label{p.tangency}
Let $K \subset \R^2$ be a centrally symmetric convex body whose boundary $\partial K$ is quasitransverse to the unit circle $S^1 = \partial E_0$.
Suppose the points of tangency are not $(\pm 1/\sqrt{2}, \pm 1/\sqrt{2})$.
Then the intersection function $I_K$ is $C^1$ at a neighborhood of $0$. 
\end{proposition}

\begin{proof} 
Assume there is at least one tangency between $\partial K$ and the unit circle $\partial E_0$, otherwise the \lcnamecref{p.tangency} follows from \cref{p.formulas}.
Fix numbers $\tau_1 < \tau_2 <\cdots < \tau_{2\ell}$ with $\tau_{i+\ell} = \tau_i+\pi$ such that the tangencies between occur at the points $(\cos \tau_i, \sin \tau_i)$. By assumption, these tangencies are quadratic and do not occur at the points $(\pm 1/\sqrt{2}, \pm 1/\sqrt{2})$.
Also fix small neighborhoods $V_i = (a_i, b_i) \ni \tau_i$. 

Define functions $G(\theta)$ and $f(\theta,t)$ as in the proof of \cref{p.formulas}, and note that $G(\theta) = 1 + f(\theta,0)$. 
These functions are continuous everywhere and are $C^2$ if $\theta$ is restricted to the set $\bigcup_i V_i$ and $t$ is close to zero.
Furthermore, for each $i$ we have $G(\tau_i)=1$, $G'(\tau_i)=0$, and $G''(\tau_i)\neq 0$; the latter inequality expresses the fact that each tangency is quadratic. 
Note also that $f_t(\tau_i, 0) \neq 0$; indeed, along the proof of \cref{p.formulas} we computed $f_t(\theta, 0) = \cos 2 \theta$, and since the tangency points are not $(\pm 1/\sqrt{2}, \pm 1/\sqrt{2})$, we have $\cos 2 \tau_i \neq 0$.

We will show that the function $A(t) \coloneqq d_\mathrm{sym}(E_t,K)$ is $C^1$ on a neighborhood of $t=0$; then it will follow from the relation \eqref{e.inclusion_exclusion} that $I_K(t)$ is also $C^1$ on a neighborhood of $t=0$.

If there are no tangencies then $A(t)$ is given by formula \eqref{e.L1}.
In order to take the tangencies into account, for each $i$ we need to add a certain correction term $C_i(t)$ to the formula.
More precisely, let 
$\varsigma_i \in \{+1,-1\}$ be the sign of $G$ on the neighborhood $V_i \ni \tau_i$, in the sense that $\varsigma_i G \ge 0$ there;
then the correction term $C_i(t)$ satisfies:
$$
\int_{a_i}^{b_i} |f(\theta,t)| \dd \theta = \varsigma_i \int_{a_i}^{b_i} f(\theta,t) \dd \theta + C_i(t) \, .
$$
Once we prove that each function $C_i$ is $C^1$ at a neighborhood of $0$, we will conclude that so are the functions $A$ and $I_K$.

For definiteness, consider the case where $G''(\tau_i)<0$ (i.e.\ $\varsigma_i=-1$) and $f_t(\tau_i, 0)>0$; the other three cases are analogous.
For each $t$ sufficiently close to zero, consider the equation $f(\theta, t) = 0$ for $\theta \in V_i$: it has no solution for $t>0$, exactly one solution $\tau_i$ for $t=0$, and exactly two solutions $\tau_i^-(t) < \tau_i^+(t)$ for $t<0$.
Then the correction term is:
$$
C_i(t) = \int_{a_i}^{b_i} \big[ |f(\theta,t)| +  f(\theta,t) \big] \dd \theta =
\begin{cases}
{\displaystyle 2\int_{\tau_i^-(t)}^{\tau_i^+(t)} f(\theta,t) \dd\theta} &\quad \text{if $t<0$}, \\
0 &\quad \text{if $t \ge 0$}.
\end{cases}
$$
For $t<0$ close to $0$, the width $\tau_i^+(t) - \tau_i^-(t)$ is $O(|t|^{1/2})$, and so $C_i(t) = O(|t|^{3/2})$.
In particular, $C_i'(0) = 0$.
Still assuming $t<0$ close to $0$, 
by Leibniz integral rule we have:
$$
\frac{1}{2} C_i'(t) 
= \int_{\tau_i^-(t)}^{\tau_i^+(t)} f_t(\theta,t) \dd \theta + \underbrace{f(\tau_i^+(t),t)}_0 \cdot (\tau_i^+)'(t) -  \underbrace{f(\tau_i^-(t),t)}_{0} \cdot (\tau_i^-)'(t) \, ,
$$
which tends to $0$ as $t \nearrow 0$.
Hence $C_i$ is a function of class $C^1$, as we wanted to show.
\end{proof}

The proof also shows that formula \eqref{e.1st} still holds in the situation of \cref{p.tangency}, but we will not use this fact.

\subsection{The key proposition}

\begin{proposition}\label{p.key}
For every $\epsilon > 0$ there exists $\delta > 0$ with the following properties.
Suppose that $K \subset \R^2$ is a centrally symmetric convex body whose boundary $\partial K$ is transverse to the unit circle $\partial E_0$, and
\begin{equation}\label{e.key_hypothesis}
\epsilon \le \area{K \cap E_0} \le \min\{ \pi, \area{K} \} - \epsilon \, .
\end{equation}
Then:
\begin{equation}\label{e.key_conclusion}
\max \big\{ |I_K'(0)| , -I_K''(0) \big\} > \delta \, ,
\end{equation}
\end{proposition}

The proof of the \cref{p.key} occupies the rest of this \lcnamecref{s.derivatives}.
Fix the convex body $K$ as above, and write $I = I_K$.

\subsection{Geometric inequalities}\label{ss.geometric}

Let us establish some preliminary inequalities.

It is convenient to reparameterize the sequence $(\xi_j)$ differently.
For each $i \in \{1,\dots,n\}$, let:
\begin{equation}\label{e.new_parameters}
\sigma_i \coloneqq \xi_{2i} + \xi_{2i-1} \, , \quad
\omega_i \coloneqq \xi_{2i} - \xi_{2i-1} \, . 
\end{equation}
So $\sigma_{i+n} = \sigma_i + 2\pi$, $\omega_{i+n} = \omega_i$, and $0 < \omega_i < \pi$.

\begin{lemma}\label{l.angle_bound}
For each $i$ we have 
$\max\{\alpha_{2i-1},\alpha_{2i}\} \le \tfrac{1}{2} \omega_i$. 
\end{lemma}

\begin{proof}
\cref{f.angle_bound} shows how to bound $\alpha_{2i}$.
The bound for $\alpha_{2i-1}$ is analogous. 
\end{proof}

\begin{center}
	\begin{figure}[hbt]
	\includegraphics[width=.4\textwidth]{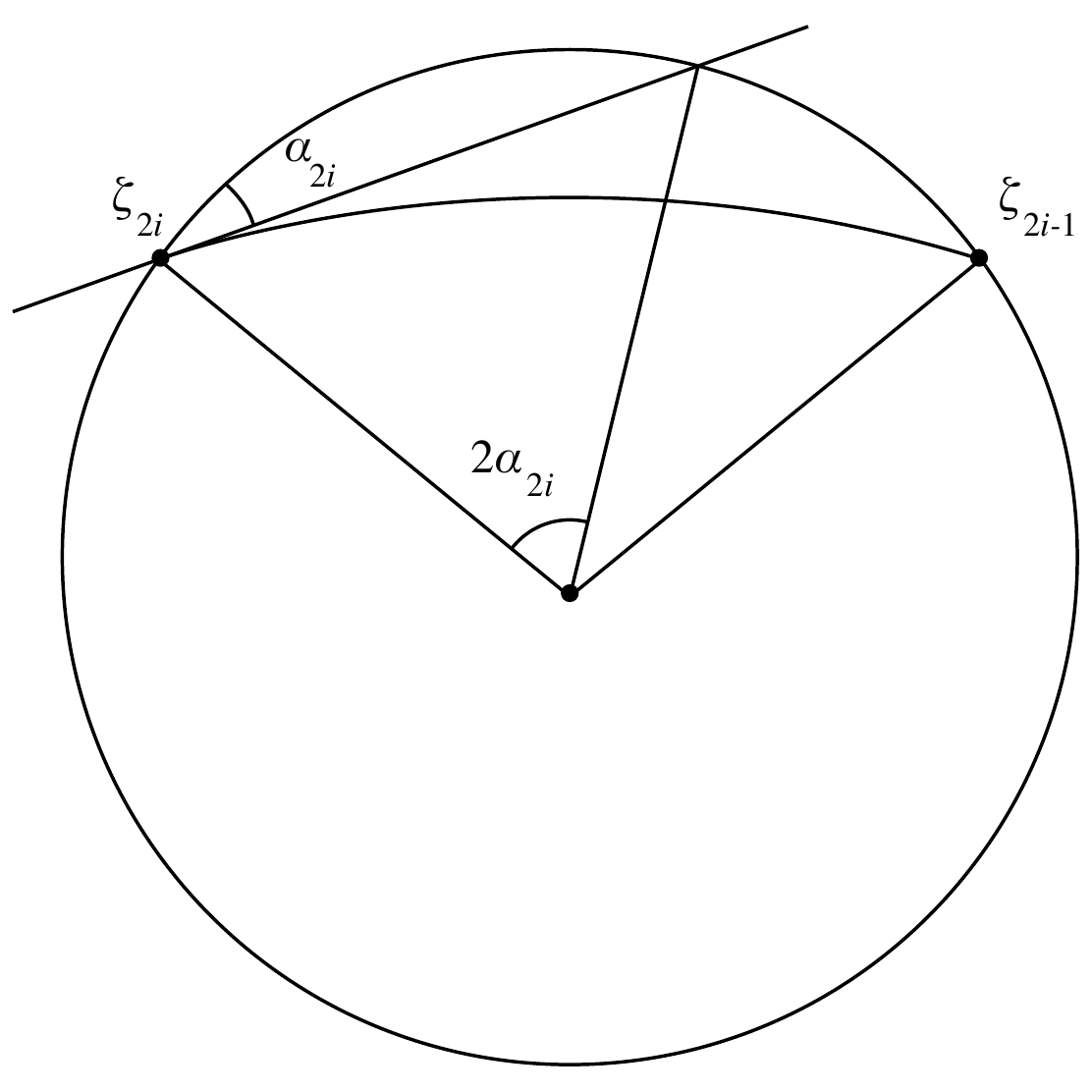}
	\caption{Proof of \cref{l.angle_bound}.}\label{f.angle_bound}
	\end{figure}
\end{center}

Next, we want some bounds on the parameters $\omega_i$.
Shifting indices if necessary, 
we assume that
$$
\omega_1 = \max\{ \omega_1, \omega_2, \dots, \omega_n\}.
$$
If $n \ge 2$, we fix $s \in \{2,\dots,n\}$ such that $\omega_s = \max\{\omega_2, \dots, \omega_n\}$.
Note that $\sum_{i=1}^n \omega_i < \pi$ and in particular 
\begin{equation}\label{e.omega_s}
\omega_s < \pi - \omega_1 \le \tfrac{\pi}{2} \, .
\end{equation}

The following \lcnamecref{l.restr} uses hypothesis \eqref{e.key_hypothesis} from \cref{p.key}, namely that $I(0) = \area{K \cap E_0}$ is not too close to $0$ nor to $\min\{\pi,\area{K}\}$.

\begin{lemma}\label{l.restr}
For every $\epsilon>0$ there exists $\kappa>0$, 
not depending on $K$, such that if 
$\epsilon \le I(0) \le \min\{\pi,\area{K}\} - \epsilon$ then:
\begin{equation}\label{e.restr1}
\kappa < \omega_1 < \pi - \kappa 
\end{equation}
and, if $n \ge 2$, 
\begin{equation}\label{e.restr2}
\omega_s < \frac{\pi}{2} - \kappa \, .
\end{equation}
\end{lemma}

\begin{proof}
Note that $K$ contains the disk of radius $\cos \tfrac{1}{2}\omega_1$ centered at the origin (see \cref{f.inner_disk}), and in particular $I(0) \ge \pi \cos^2 \tfrac{1}{2}\omega_1$. 
By assumption, $I(0) \le \pi - \epsilon$, and so $\omega_1$ cannot be too small, 
proving the first inequality in \eqref{e.restr1}.

\medskip

Now consider \eqref{e.restr2}: if this inequality does not hold then, by \eqref{e.omega_s}, both $\omega_1$ and $\omega_s$ are approximately $\frac{\pi}{2}$. Then $K\setminus E_0$ is contained in the union of the four small regions represented in \cref{f.4_corners}. This contradicts the fact that $\area{K\setminus E_0}  = \area{K} - I(0) \ge \epsilon$ is not too small.

\begin{figure}[htb]
    \centering
    \begin{minipage}{0.4\textwidth}
        \centering
        \includegraphics[width=0.9\textwidth]{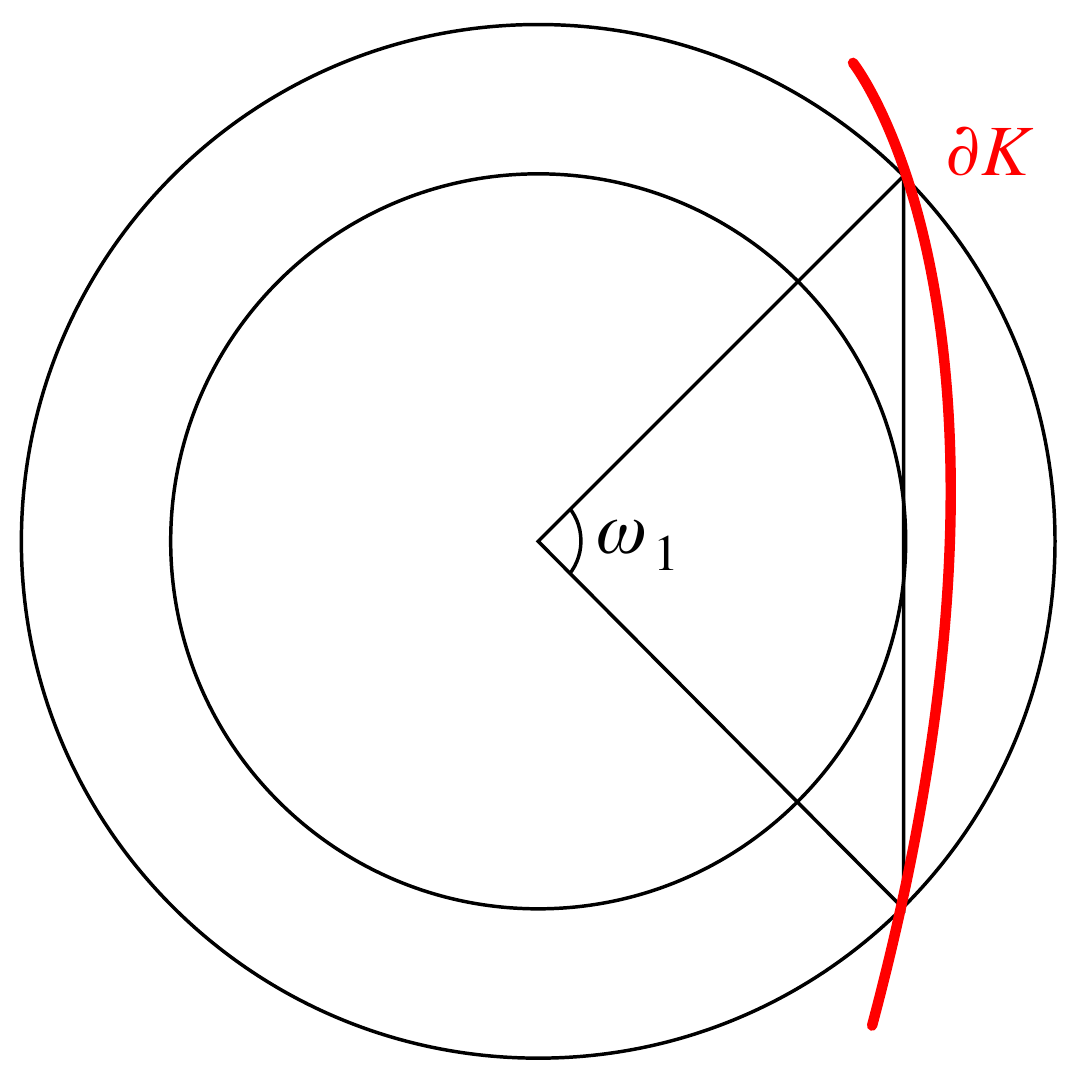}
        \caption{Proof of the first inequality in \eqref{e.restr1}.}\label{f.inner_disk}
    \end{minipage}\hfill
    \begin{minipage}{0.5\textwidth}
        \centering
        \includegraphics[width=0.9\textwidth]{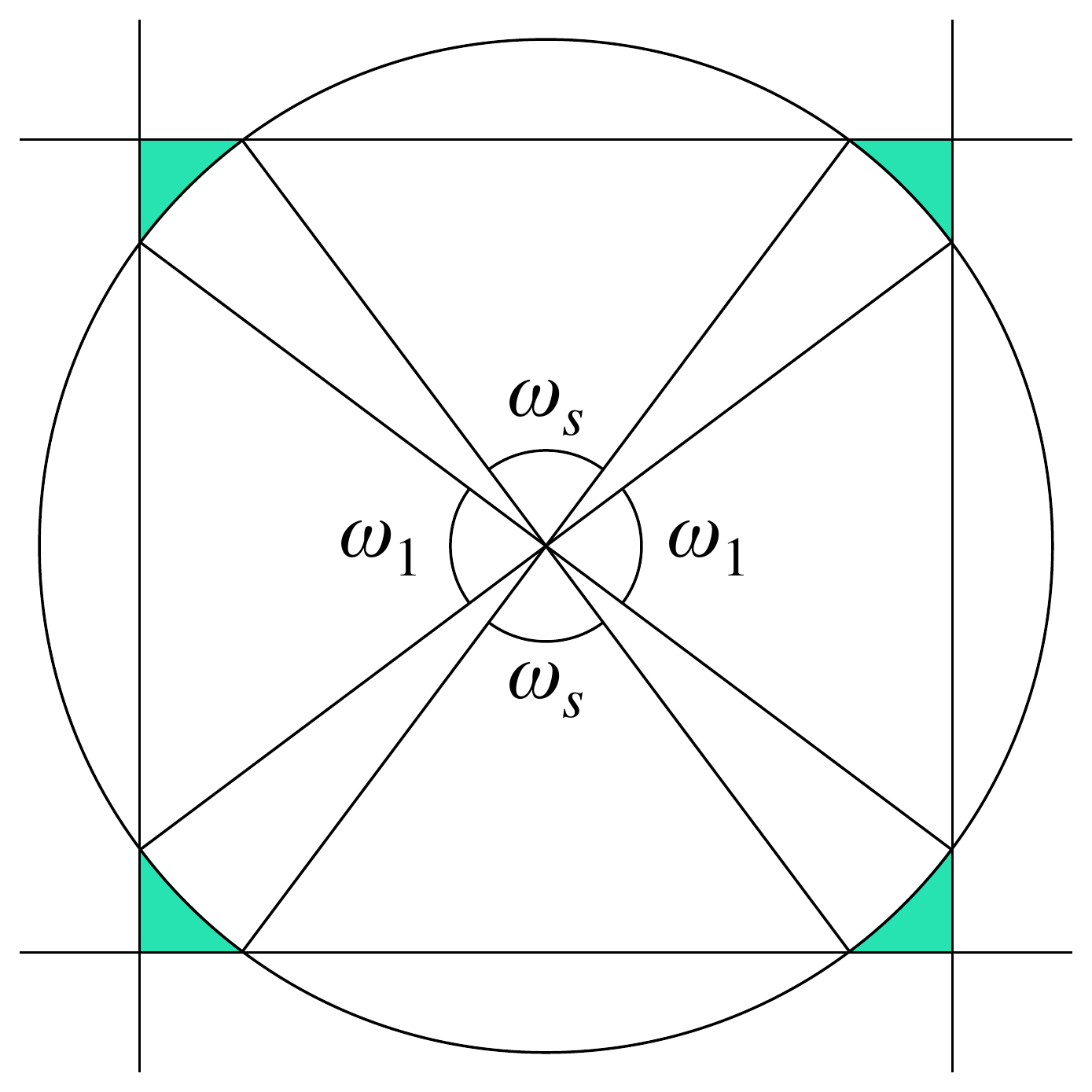}
        \caption{Proof of \eqref{e.restr2}.}\label{f.4_corners}
    \end{minipage}
\end{figure}

\medskip

The second inequality in \eqref{e.restr1} is the trickiest one. 
Let $L_1$ (resp.\ $L_2$)  be the tangent line to $\partial K$ at the point $\zeta_1$ (resp.\ $\zeta_{2n+2}$), oriented so that $K$ sits to the left of this line. 
The lines $L_1$ and $L_2$ cross the circle $\partial E_0$ forming angles $\alpha_1$ and $\alpha_{2n+2} = \alpha_2$, respectively. 
Let $R_i$ be the part of the disk $E_0$ to the right of the line $L_i$: see \cref{f.pogoball}.

\begin{center}
	\begin{figure}[hbt]
	\includegraphics[width=\textwidth]{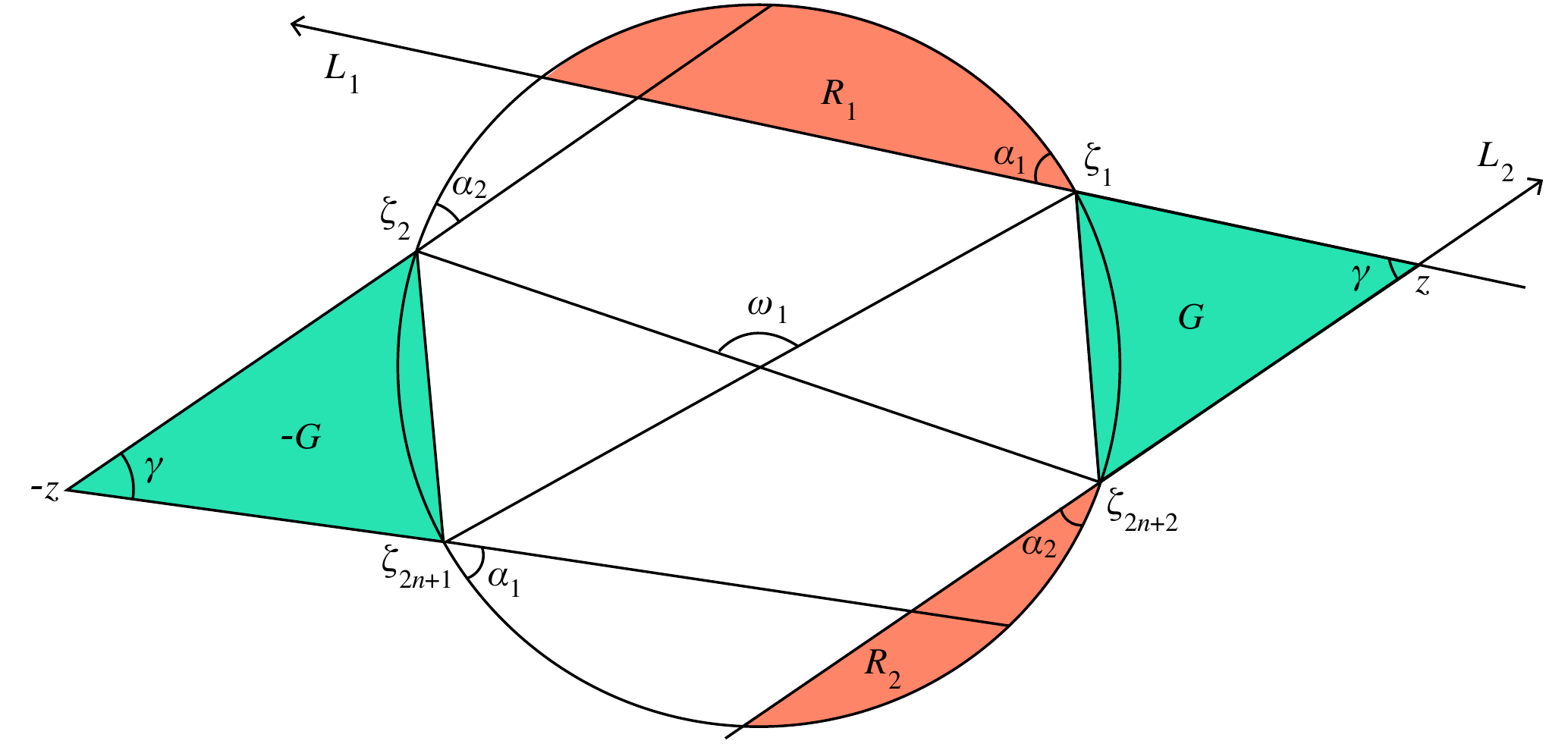}
	\caption{Proof of the second inequality in \eqref{e.restr1}.}\label{f.pogoball}
	\end{figure}
\end{center}

The regions $R_1$ and $R_2$ are disjoint and their interiors are contained in $E_0 \setminus K$.
In particular,
$$
\area{R_1} + \area{R_2} \le \area{E_0 \setminus K}  = \pi - I(0) \le \pi - \epsilon.
$$
So the areas $\area{R_1}$ and $\area{R_2}$ cannot be both too close to $\tfrac{\pi}{2}$.
On the other hand, these areas are related to the crossing angles $\alpha_i$ as follows:
$$
\area{R_i} = \alpha_i - \tfrac{1}{2}\sin 2\alpha_i \, .
$$
Therefore the angles $\alpha_1$ and $\alpha_2$ cannot be both too close to $\tfrac{\pi}{2}$.
By \cref{l.angle_bound}, we have $\max\{\alpha_1,\alpha_2\} \le \tfrac{1}{2} \omega_1$, and in particular the quantity $\gamma \coloneq \omega_1 - \alpha_1 - \alpha_2$ is nonnegative.
If $\gamma$ is zero or small then $\omega_1$ is not too close to $\pi$, as desired. 
So assume from now on that $\gamma$ is not too close to $0$.
Then the lines $L_1$ and $L_2$ cannot be parallel; indeed they cross forming angle $\gamma$ at some point $z$.
Recall that the centrally symmetric convex body $K$ sits to the left of each oriented line $L_1$ and $L_2$; furthermore, the arc of the counterclockwise-oriented Jordan curve $\partial K$ from $\zeta_1$ to $\zeta_2$ is contained in the disk $E_0$.
It follows from these observations that $K \subseteq E_0 \cup G \cup (-G)$, where 
$G$ is the (filled) triangle with vertices $\zeta_1$, $\zeta_{2n+2}$, $z$.
In particular,
$$
2\area{G} = \area{G \cup (-G)} \ge \area{K \setminus E_0} = \pi - I(0) \ge \epsilon \, ,
$$
Note that the triangle $G$ has a side $[\zeta_1, \zeta_{2n+2}]$ of length $\ell \coloneq 2 \cos \tfrac{1}{2} \omega_1$, and therefore its area cannot exceed the area of an isosceles triangle with angle $\gamma$ and opposite side $\ell$, that is, 
\begin{equation*}
\area{G} \le \tfrac{1}{4} \ell^2 \cot\tfrac{\gamma}{2}  \, .
\end{equation*}
Since $\gamma$ and $\area{G} \ge \tfrac{\epsilon}{2}$ are bounded away from $0$, so is $\ell$.
It follows that $\omega_1$ cannot be too close to $\pi$.
This completes the proof of the second inequality in \eqref{e.restr1} and of the \lcnamecref{l.restr}.
\end{proof}

\subsection{More manipulation of the derivatives}

We will now come back to the formulas obtained in \cref{p.formulas} and rewrite them in terms of the new parameters \eqref{e.new_parameters}; we will also use \cref{l.angle_bound} to obtain a convenient lower bound for minus the second derivative.

\begin{lemma}\label{l.new}
\begin{align}
I'(0)  	&= \sum_{i=1}^{n}\sin \omega_i \cos \sigma_i \, , \label{e.new_1st} \\
-I''(0) &\ge \sum_{i=1}^{n} \left[ \sin \omega_i \sin^2 \sigma_i + (\cot \tfrac{1}{2}\omega_i - \sin \omega_i) \cos^2 \sigma_i \right] \, . \label{e.new_2nd}
\end{align}
\end{lemma}

\begin{proof} 
We can rewrite \eqref{e.1st} as
$$
I'(0)  
= \frac{1}{2} \sum_{i=1}^{n} ( \sin 2\xi_{2i} - \sin 2\xi_{2i-1})  \, , 
$$
so \eqref{e.new_1st} follows from the identity
$$
\sin x - \sin y = 2  \sin \frac{x-y}{2}  \cos \frac{x+y}{2}
$$
(together with the definitions \eqref{e.new_parameters}).
Analogously, rewriting \eqref{e.2nd} as
$$
-I''(0) = \frac{1}{4} \sum_{i=1}^{n} \left[ \sin 4\xi_{2i} - \sin 4\xi_{2i-1}
+ \frac{1+ \cos 4\xi_{2i}}{\tan \alpha_{2i}} + \frac{1 + \cos 4\xi_{2i-1}}{\tan \alpha_{2i-1}} \right] \, . 
$$
By \cref{l.angle_bound}, $\max(\alpha_{2i},\alpha_{2i-1}) \le \tfrac{1}{2} \omega_i$.
So, using the identity
$$
\cos x + \cos y = 2  \cos \frac{x-y}{2}  \cos \frac{x+y}{2}
$$
we obtain:
$$
-I''(0) \ge \frac{1}{2} \sum_{i=1}^{n} \left[ \sin 2\omega_i \cos 2\sigma_i + \cot \tfrac{1}{2}\omega_i \, \Bigl(1 + \cos 2\omega_i \cos 2\sigma_i \Bigr) \right] \, . 
$$
Let us manipulate the quantity between square brackets.
For simplicity of writing, we omit the $i$ indices:
\begin{align*}
[\dots]
&= \underbrace{\sin 2\omega}_{2 \sin \omega \cos \omega} \cos 2\sigma + \cot \tfrac{1}{2}\omega \, \Bigl[ \underbrace{(-1 + \cos 2\omega)}_{-2\sin^2 \omega}\cos 2\sigma +  \underbrace{(1 + \cos 2\sigma)}_{2 \cos^2 \sigma} \Bigr]
\\ 
&= 2 \sin \omega \underbrace{\Bigl( \cos \omega - \cot \tfrac{1}{2}\omega \, \sin \omega\Bigr)}_{1} \underbrace{\vphantom{\Bigl(\Bigr)} \cos 2\sigma}_{\cos^2 \sigma - \sin^2 \sigma} + \  2 \cot \tfrac{1}{2}\omega \cos^2 \sigma 
\\
&= 2\sin\omega \sin^2 \sigma + 2 \left( \cot \tfrac{1}{2}\omega - \sin\omega \right) \cos^2 \sigma \, ,
\end{align*}
yielding \eqref{e.new_2nd}.
\end{proof}

\subsection{Proof of the key Proposition~\ref{p.key}}\label{ss.key_proof}

The proof is a case-by-case analysis; in most of the cases we will show that $I''(0)$ is negative and away from zero, but in a few cases the conclusion is that $I'(0)$ is away from zero.
All the estimates on those derivatives will be obtained from \cref{l.new}, which will not be explicitly mentioned each time.
All estimates are explicit and ultimately we will obtain a lower bound for $\max\{|I'(0)|,-I''(0)\}$ that depends only on $\kappa$ from \cref{l.restr}, and therefore is a function of $\epsilon$ which is independent of $K$.

Let us introduce some notation:
\begin{align*}
f(\omega)			&\coloneqq \cot \tfrac{1}{2}\omega - \sin \omega 
	 				= \cot \omega - \csc \omega - \sin \omega \, , \\
g(\omega,\sigma)	&\coloneqq \sin \omega \sin^2 \sigma + f(\omega) \cos^2 \sigma \, .
\end{align*}
So the fundamental inequality \eqref{e.new_2nd} can be rewritten as:
\begin{equation}\label{e.gatito}
-I''(0) \ge \sum_{i=1}^n g(\omega_i,\sigma_i) \, .
\end{equation}
Note that $g(\omega,\sigma)$ is a convex combination of the two functions $\sin \omega$ and $f(\omega)$, which are  plotted in \cref{f.2_functions}.
We will use this fact repeatedly to obtain bounds.  
Note that the abscissa of the crossing between the two graphs is $\tfrac{\pi}{3}$, that 
$\min g = \min f  
> -0.31$, 
and that
\begin{equation}\label{e.minus_half}
f(\omega) \ge - \tfrac{1}{2} \sin \omega \, \quad \forall \omega \in (0,\pi) \, .
\end{equation}
Recall that $\omega_1$ is the biggest of all angles $\omega_i$'s and so it is the only angle that can be bigger than $\tfrac{\pi}{2}$; therefore the sum in \eqref{e.gatito} contains at most one negative term.

\begin{center}
	\begin{figure}[hbt]
	\includegraphics[width=.85\textwidth]{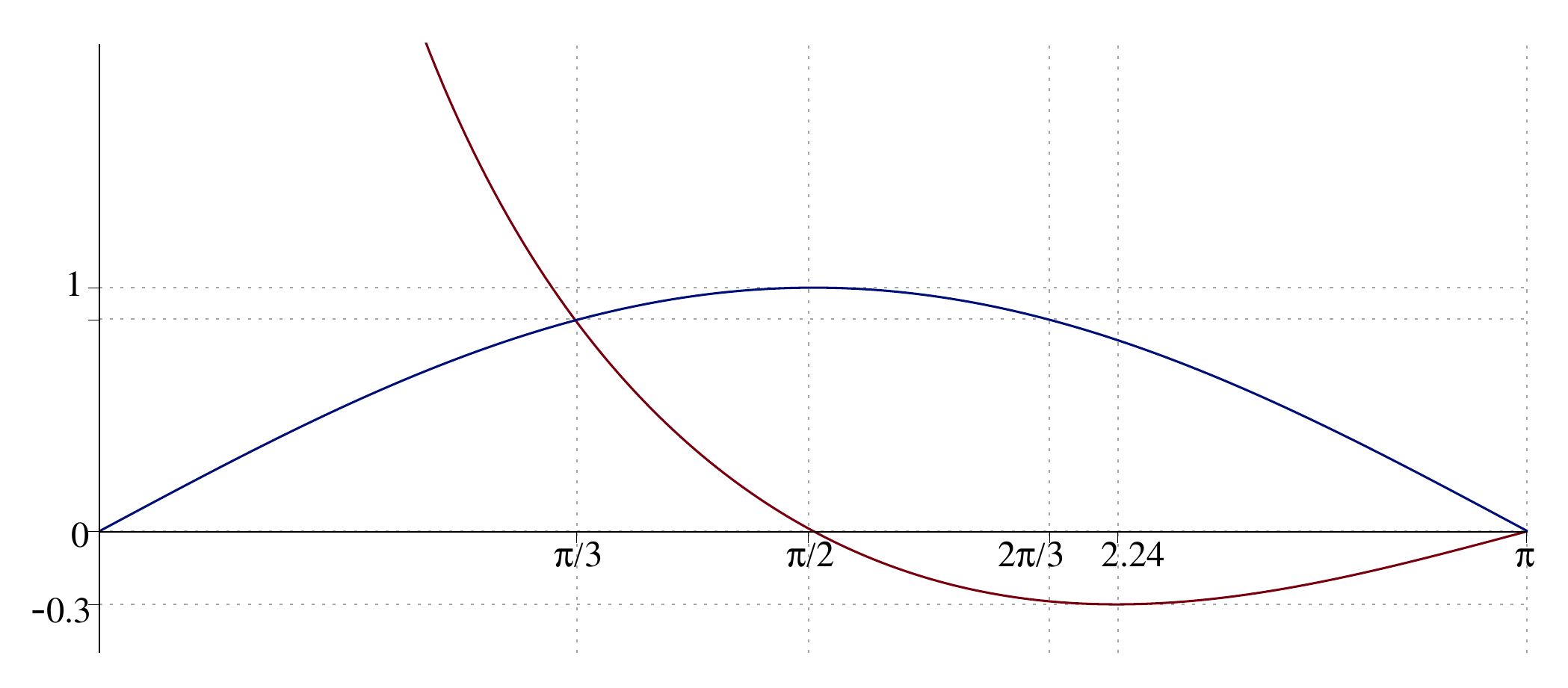}
	\caption{Plot of the functions $\sin$ and $f$.}\label{f.2_functions}
	\end{figure}
\end{center}

\medskip\noindent\textbf{Case 1.} 
$|\cos \sigma_1| \le 0.8$.
Then:
\begin{align*}
-I''(0)
&\ge g(\omega_1,\sigma_1)  \\
&=   \sin \omega_1 \sin^2 \sigma_1 + f(\omega_1) \cos^2 \sigma_1  \\
&\ge \sin \omega_1 \left( \sin^2 \sigma_1  - \tfrac{1}{2} \cos^2 \sigma_1 \right) &\quad&\text{(by \eqref{e.minus_half})} \\
&\ge \sin \omega_1 \left( 0.36 - \tfrac{1}{2} 0.64 \right) \\
&\ge 0.02 \sin \kappa &\quad&\text{(by the first inequality in \eqref{e.restr1})} \\
&>0 \, .
\end{align*}

\medskip

In the remaining cases, we assume $\boxed{|\cos \sigma_1| > 0.8}$.

\medskip\noindent\textbf{Case 2.} $n=1$.
Then, by the first inequality in \eqref{e.restr1},
$$
|I'(0)| = \sin \omega_1 |\cos\sigma_1| > 0.8 \sin \kappa > 0 \, .
$$

\medskip

In the remaining cases, we assume $\boxed{n \ge 2}$.

\medskip\noindent\textbf{Case 3.} $\omega_1 \le \tfrac{3\pi}{8}$.  
Then, by the first inequality in \eqref{e.restr1}, 
$$
-I''(0)
\ge g(\omega_1,\sigma_1) 
\ge \min \bigr\{ \sin \omega_1, f(\omega_1) \bigl\}
\ge \min \bigr\{ \sin \kappa, f(\tfrac{3\pi}{8}) \bigl\} > 0 \, .
$$

\medskip

In the remaining cases, we assume $\boxed{\omega_1 > \tfrac{3\pi}{8}}$.
Recall from \Cref{ss.geometric} that $\omega_s$ is the second biggest of the $\omega_i$'s.

\medskip\noindent\textbf{Case 4.} $\omega_s \ge \tfrac{\pi}{3}$.
So:
\begin{alignat*}{2}
-I''(0) 
&\ge g(\omega_1,\sigma_1) + g(\omega_s,\sigma_s)\\
&\ge f(\omega_1) + f(\omega_s)					
&\quad&\text{(since $\omega_1 \ge \omega_s \ge  \tfrac{\pi}{3}$)} 
\end{alignat*}
Now, $\omega_1 < \pi-\omega_s \le \tfrac{2\pi}{3}$ and the function $f$ is decreasing on the interval $(0,\tfrac{2\pi}{3}]$ (actually, it is decreasing on a slightly bigger interval), so:
$$
-I''(0) \ge f(\pi-\omega_s) + f(\omega_s) = 2 (\csc - \sin)(\omega_s).
$$
Now, using \eqref{e.restr2} we obtain:
$$
-I''(0) \ge 2 (\csc - \sin)(\tfrac{\pi}{2} - \kappa) > 0
$$
and we are done in this case.
	
\medskip

In the remaining cases, we assume $\boxed{ \omega_s < \tfrac{\pi}{3} }$.
Define the following numbers:
\begin{align*}
\kappa' &\coloneqq 0.1 \sin \kappa \, , \\
\Lambda	&\coloneqq \max\big\{ -g(\omega_1,\sigma_1) , 0 \big\} + \kappa' \, , \\ 
\Sigma	&\coloneqq \sum_{i=2}^n g(\omega_i,\sigma_i) \, .
\end{align*}

\medskip\noindent\textbf{Case 5.} $\Sigma \ge \Lambda$.
Then:
\begin{align*}
-I''(0) 
&=		\Sigma + g(\omega_1,\sigma_1)  \\
&\ge	\Sigma + \kappa' - \Lambda  \\
&\ge 	\kappa' 
\end{align*}
and we are done. 

\medskip

In the final and most interesting case, we assume $\boxed{\Sigma < \Lambda}$.

\medskip\noindent\textbf{Case 6.}
Let us establish two upper estimates for $\Lambda$;
the first one is:
\begin{equation}\label{e.kappa1}
\Lambda \le \max(-g) + \kappa' < 0.31 + \kappa' < 0.41 \, ,
\end{equation}
and the second one is:
\begin{alignat}{2}
\Lambda  &\le -g(\omega_1,\sigma_1) + \kappa' 	\notag \\
		&\le -f(\omega_1) + \kappa'				&\quad&\text{(since $\omega_1 > \tfrac{3\pi}{8} > \tfrac{\pi}{3}$)} \notag\\
		&\le \tfrac{1}{2} \sin \omega_1 + \kappa' &\quad&\text{(by \eqref{e.minus_half}).} \label{e.kappa2}
\end{alignat}

Define:
$$
\Delta \coloneqq \arcsin \Lambda,
$$
which by \eqref{e.kappa1}, satisfies $\Delta < 0.43$.

Note that:
$$
\sin \Delta = \Lambda > \Sigma \ge g(\omega_s,\sigma_s) \ge \sin \omega_s \qquad \text{(since $\omega_s < \tfrac{\pi}{3}$)},
$$
that is,  $\boxed{\omega_s \le \Delta}$.

Let $\varphi(\omega) \coloneqq f(\omega) - \sin(\omega)$.
We can rewrite $\Sigma$ as:
$$
\Sigma = \sum_{i=2}^n \left[ \sin \omega_i + \varphi(\omega_i) \cos^2 \sigma_i \right] \, .
$$
Since $\omega_s \le \Delta < \tfrac{\pi}{3}$ and the function $\varphi$ is positive and decreasing on the interval $(0,\tfrac{\pi}{3})$, the assumption $\Sigma < \Lambda$ yields two other inequalities:
\begin{align*}
\sum_{i=2}^n \sin \omega_i &< \Lambda  \, , \\
\sum_{i=2}^n \cos^2 \sigma_i &< \frac{\Lambda}{\varphi(\Delta)} \, .
\end{align*}
So, on one hand,
$$
\sum_{i=2}^n \sin^2 \omega_i 
\le \sin \omega_s \sum_{i=2}^n \sin \omega_i
\le \Lambda^2  \, .
$$
On the other hand, recalling that $\Lambda = \sin \Delta$ and $\varphi(\Delta) = \cot \Delta + \csc \Delta - 2 \sin \Delta$,
$$
\sum_{i=2}^n \cos^2 \sigma_i 
< \frac{\Lambda}{\varphi(\Delta)} 
=   \frac{\Lambda^2}{\cos \Delta + 1 - 2 \sin^2 \Delta} 
\le \frac{\Lambda^2}{\cos 0.43 + 1 - 2 \sin^2 0.43}
< 0.65 \Lambda^2 \, . 
$$

By Cauchy--Schwarz inequality,
$$
\left| \sum_{i=2}^n \sin \omega_i \cos \sigma_i \right| 
\le  \sqrt{ \sum_{i=2}^n \sin^2 \omega_i \times \sum_{i=2}^n \cos^2 \sigma_i}
< \sqrt{\Lambda^2 \times 0.65 \Lambda^2}
< 0.81 \Lambda^2 
< \Lambda \, .
$$
This inequality together with \eqref{e.kappa2} allows us to show that $A'(0)$ is not too close to zero:
\begin{align*}
|I'(0)|
&\ge  \bigl| \sin \omega_1 \cos \sigma_1 \bigr| - \left| \sum_{i=2}^n \sin \omega_i \cos \sigma_i \right| \\
&\ge 0.8 \sin \omega_1 - \Lambda \\
&\ge 0.3 \sin \omega_1 -  \kappa' \\
&\ge 0.2 \sin \kappa  \\
&>0 \, .
\end{align*}
This concludes the proof of \cref{p.key}.

\section{Proof of the quasiconcavity Theorem~\ref{t.qc}} \label{s.implication}

\subsection{Setting up the proof}\label{ss.setup}

Let us say that a centrally symmetric body $K \in \cC^2$ is \emph{regular} if it satisfies the following conditions:
\begin{enumerate}
\item\label{i.reg1}
the boundary $\partial K$ is a $C^2$ curve;
\item\label{i.reg2}
there is a finite set $T \subset \R$ such that for every $t \in \R \setminus T$, the curves $\partial K$ and $\partial E_t$ are transverse;
\item\label{i.G3}
for every $t \in \R$, the curves $\partial K$ and $\partial E_t$ are quasitransverse, and 
the points of (necessarily quadratic) tangency do not belong to the envelope hyperbolas $xy = \pm 1/2$ of the family of curves $(\partial E_t)_{t \in \R}$.
\end{enumerate}

We will prove that regularity is dense in $\cC^2$; actually we will show more: 

\begin{proposition}[Regularization]\label{p.regularization}
For every $K \in \cC^2$ and every $\epsilon>0$ there exists a regular $\tilde K \in \cC^2$ such that $\area{\tilde K} = \area{K}$ and $d_\mathrm{sym}(\tilde K, K) < \epsilon$.
\end{proposition}

On the other hand, using \cref{p.formulas,p.tangency,p.key} one can check that \cref{t.qc} holds for regular convex bodies in $\cC^2_\pi$, that is, the associated intersection functions are strictly quasiconcave.
Actually, the uniformity provided by \cref{p.key} will allow us to prove a more precise property:

\begin{proposition}[Quantitative quasiconcavity]\label{p.new_qqc}
Given $\epsilon > 0$ and $r > 0$, there exists $\eta > 0$ with the following properties.
For every regular $K \in \cC^2_\pi$, if $t_0$, $t_1$, $t_2 \in \R$ are such that: 
\begin{equation*}
t_0 + r \le t_1 \le t_2 - r \quad \text{and} \quad
2\epsilon \le I_K(t_1) \le \max\{\pi,\area{K}\} - 2\epsilon \, ,
\end{equation*}
then:
\begin{equation*}
I_K(t_1) > \min\{I_K(t_0), I_K(t_2)\} + \eta \, .
\end{equation*}
\end{proposition}

Let us postpone the proofs of \cref{p.regularization,p.new_qqc}, and use them to deduce the \lcnamecref{t.qc}:

\begin{proof}[Proof of \cref{t.qc}]
Fix $K \in \cC^2_\pi$ and arbitrary numbers $t_0 < t_1 < t_2$. 
Let 
$$
r \coloneqq \min\big\{t_1 - t_0, t_2-t_1\big\} \quad \text{and} \quad
\epsilon \coloneqq \tfrac{1}{3}\min\big\{I_K(t_1), \max\{\pi,\area{K}\} - I_K(t_1) \big\} \, .
$$
Let $\eta = \eta(\epsilon,r)$ be given by \cref{p.new_qqc}.
Reducing $\eta$ if necessary, we assume $0 < \eta \le 4\epsilon$.
By \cref{p.regularization}, there exists a regular body $\tilde K \in \cC^2$ with the same area as $K$ such that $d_\mathrm{sym}(\tilde K, K)< \frac{\eta}{2}$.
Recalling that $\cC^2_\pi$ is open in $\cC^2$, we can assume that $\tilde K \in \cC^2_\pi$.
As a consequence of relation \eqref{e.inclusion_exclusion}, for every $t \in \R$ we have $\big| I_{\tilde K} (t) - I_K(t)\big| < \frac{\eta}{4} \le \epsilon$. In particular,
\begin{align*}
I_{\tilde K} (t_1) &\ge I_K(t_1) - \epsilon \ge 3\epsilon - 2\epsilon \, , \\ 
I_{\tilde K} (t_1) &\le I_K(t_1) + \epsilon \le \max\{\pi,\area{\tilde K}\} -3\epsilon + \epsilon  \, . 
\end{align*}
This allows us to apply \cref{p.new_qqc} to the convex body $\tilde K$ and obtain 
$I_{\tilde K} (t_1) \ge \min\{I_{\tilde K}(t_0) , I_{\tilde K}(t_1) \} + \eta$. 
It immediately follows that 
$I_K (t_1) \ge \min\{I_K(t_0) , I_K(t_1) \} + \frac{\eta}{2}$. 
This proves that the function $I_K$ is quasiconcave.
\end{proof}

\subsection{Proof of the regularization Proposition~\ref{p.regularization}}\label{ss.regularization}

Let $\P^1$ denote the projective space of $\R^2$, i.e.\ the set of all lines through the origin. Let $[v]\in \P^1$ denote the line determined by a nonzero vector $v \in \R^2_*$.

If $\Gamma \subset \R^2$ is any smooth $1$-dimensional submanifold, 
denote by $\hat{\Gamma} \subset \R^2 \times \P^1$ the set of pairs $(u,[v])$ such that $u \in \Gamma$ and $v$ is tangent to $\Gamma$ at $u$.
Define the following sets:
\begin{alignat*}{2}
Z_1 &\coloneqq \hat{H}  &\quad &\text{where $H$ is the pair of hyperbolas $xy  = \pm 1/2$;}  \\ 
Z_2 &\coloneqq \bigcup_{t\in\R} \widehat{\partial E_t} &\quad &\text{where $E_t$ are the ellipses~\eqref{e.standard_family}.}
\end{alignat*}
The latter union is disjoint, because any two distinct ellipses in our family have transverse boundaries. 

\begin{lemma}\label{l.submanifolds}
$Z_1$ and $Z_2$ are closed smooth submanifolds of $\R^2 \times \P^1$ of respective dimensions $1$ and $2$, and $Z_1 \subset Z_2$.
\end{lemma}

The \lcnamecref{l.submanifolds} is intuitively clear, 
but for completeness we provide a proof at the end of this \lcnamecref{ss.regularization}.

Let $\T \coloneqq \R / 2\pi\Z$, the additive group of real numbers mod $2\pi$.
A \emph{regular parametrization} of a smooth Jordan curve $\Gamma \subset \R^2$ is a map $g \colon \T \to \R^2$ that is a smooth diffeomorphism onto $\Gamma$.
In that case, let $\hat{g} \colon \T \to \R^2 \times \P^1$ denote the map $\hat{g}(\theta) \coloneqq (g(\theta), [g'(\theta)])$, which is a smooth diffeomorphism onto $\hat{\Gamma}$.

\begin{lemma}\label{l.transversalities}
Suppose $K \in \cC^2$ has smooth boundary, and $g \colon \T \to \R^2$ is a regular parametrization of it.
If $\hat{g}$ is transverse to both submanifolds $Z_1$ and $Z_2$ then the body $K$ is regular.
\end{lemma}

\begin{proof}
Let $K \in \cC^2$ and suppose that $\partial K$ has a regular parametrization $g$ such that $\hat{g}$ is transverse to both $Z_1$ and $Z_2$. 
The first regularity condition (\ref{i.reg1}) is automatic: the boundary $\partial K$ is actually smooth.

Since the ambient space $\R^2 \times \P^1$ is $3$-dimensional, transversality implies that there are finitely many (if any) parameters $\theta_i$ such that the point $\hat{g}(\theta_i)$ belongs to the surface $Z_2$. Each of these points belongs to a unique curve $\widehat{\partial E_{t_i}}$.
Let $T \subset \R$ be the set of the $t_i$'s.
If $t \not \in T$ then the image of $\hat{g}$ does not intersect the curve $\widehat{\partial E_t}$, which means that the plane curves $\partial K$ and $\partial E_t$ are transverse.
This shows that $K$ meets regularity condition (\ref{i.reg2}).

On the other hand, for each $i$, the plane curves $\partial K$ and $\partial E_t$ are tangent at the point $g(\theta_i)$.
Suppose for a contradiction that this tangency is \emph{not} quadratic, i.e., the curves have a second-order contact.
Choose a regular parametrization $h_i \colon \T \to \R^2$ of $\partial E_{t_i}$ such that $\hat{h_i}(\theta_i) = \hat{g}(\theta_i)$. 
Then parameterized curves $\hat{g}$ and $\hat{h_i}$ have a first-order contact (i.e.\ are tangent) at parameter $\theta_0$.
Since $\hat{h_i}$ is an immersion whose image is contained in the surface $Z_2$, we conclude that $\hat{g}$ is \emph{not} transverse to $Z_2$, which is a contradiction. 
We have shown that the the tangencies between the plane curves $\partial K$ and $\partial E_{t_i}$ are all quadratic, i.e., the curves are quasitranverse.

Furthermore, the fact that the mapping $\hat{g}$ is transverse to the $1$-dimensional submanifold $Z_1 \subset Z_2$ means that its image does not intersect $Z_1$. That is, all tangency points $g(\theta_i)$ are outside the forbidden hyperbolas $xy=\pm 1/2$. 
This concludes the proof that the body $K$ is regular.
\end{proof}

\begin{proposition}\label{p.transSL2R}
If $K \in \cC^2$ has smooth boundary then 
there is an open dense subset $U$ of $\SL(2,\R)$ such that 
if $L \in U$ then the body $LK$ is regular.
\end{proposition}

\begin{proof}
Let $Y \subset \R^2 \times \P^1$ be the set of pairs $(u,[v])$ such that $u$ and $v$ are linearly independent. Note that $Z_1$, $Z_2$ are subsets of $Y$.
The group $\SL(2,\R)$ acts on $Y$ in the obvious way: $L(u,[v]) = (Lu, [Lv])$.
This action is smooth, transitive, and faithful; in particular $\SL(2,\R)$ and $Y$ are diffeomorphic.

Let $g \colon \T \to \R^2$ be a regular parametrization of $\partial K$.
Since $K$ is centrally symmetric, $\hat{g}$ takes values in $Y$.
The map $f \colon \T \times \SL(2,\R) \to Y$ defined by
$f(\theta, L) \coloneqq L \hat{g}(\theta)$
is a submersion.
Therefore, by the transversality theorem \cite[p.~68]{GP} (or see \cite[Theorem~2.7]{Hirsch} for a more precise version),
the set $U \subset \SL(2,\R)$ formed by those $L$ such that $f(\mathord{\cdot}, L) \colon \T \times \SL(2,\R) \to Y$ is transverse to $Z_1$ and to $Z_2$
is open and dense in $\SL(2,\R)$.
Take $L \in U$.
Noting that $f(\mathord{\cdot}, L)  = \widehat{L \circ g}$, it follows from \cref{l.transversalities} that the body $LK$ is regular.
\end{proof}

The previous \lcnamecref{p.transSL2R} implies the result we are looking for:

\begin{proof}[Proof of \cref{p.regularization}]
Given $K \in \cC^2$, we initially perturb it so that the area is unchanged and the boundary becomes smooth: for example, we can take an inscribed polygon, smoothen the corners, and inflate it to recover the area.  $\cC^2_\pi$ is open in $\cC^2$. 
Then by \cref{p.transSL2R} we can apply a element of $\SL(2,\R)$ close to the identity and so obtain the desired regular body approximating $K$.
\end{proof}

Finally, we check that $Z_1$ and $Z_2$ are indeed submanifolds.

\begin{proof}[Proof of \cref{l.submanifolds}]
The \emph{QR decomposition} comes in handy: there is a diffeomorphism $\Phi \colon \R \times \R \times \T \to \SL(2,\R)$ given by:
$$
\Phi(\tau, \rho , \xi)   
\coloneqq 
\begin{pmatrix}
	e^{-\tau/2}	& \rho \\
	0			& e^{\tau/2}
\end{pmatrix}
\begin{pmatrix}
	\cos \xi	& -\sin\xi \\
	\sin \xi	& \cos\xi
\end{pmatrix}
\, .
$$
Changing coordinates under $\Phi$, 
the left action of the diagonal subgroup corresponds to translation of the first coordinate.

Let $Y$ be as in the proof of \cref{p.transSL2R}.
Define a diffeomorphism $\Psi \colon \SL(2,\R) \to Y$ by $\Psi(L) \coloneqq (Le_1, [Le_2])$, where 
$\{e_1,e_2\}$ is the canonical basis of $\R^2$.
So the map $\Psi \circ \Phi$ allows us to put global coordinates $(t, \rho , \xi)$ on $Y$.

Note that, in the coordinates just described, $\widehat{\partial E_0}$ is given by equations $t = 0$, $\rho = 0$.
So, applying the diagonal subgroup, we conclude that $Z_2$ is the surface $\rho=0$. 
Analogously, $Z_1$ corresponds to $\rho = 0$ and $\xi \in \big\{\pm \tfrac{\pi}{4} , \pm \tfrac{3\pi}{4} \big\}$.
This proves \cref{l.submanifolds}.
\end{proof}


\subsection{Proof of the quantitative quasiconcavity Proposition~\ref{p.new_qqc}}

Let us begin by collecting the more direct consequences of \cref{p.formulas,p.tangency,p.key} in the following:

\begin{lemma}\label{l.collect}
Let $K \in \cC^2_\pi$ be regular, and let $T \subset \R$ be the corresponding (finite) set of tangency parameters.
Then the intersection function $I = I_K$ has the following properties:
\begin{enumerate}

\item\label{i.collect1}
$I$ is of class $C^1$.

\item\label{i.collect2}
The restriction of $I$ to the set $\R \setminus T$ is of class $C^2$.

\item\label{i.collect3}
For every $t \in \R \setminus T$ and $\epsilon>0$ we have:
$$
\epsilon \le I(t) \le \min\{\pi,\area{K} \} - \epsilon \quad \Rightarrow \quad
\max \big\{ |I'(t)| , -I''(t) \big\} > \delta(\epsilon) \, ,
$$
where $\delta(\mathord{\cdot})$ is the function from \cref{p.key}.

\item\label{i.collect4}
$I$ has a unique critical point $t_*$.

\item\label{i.collect5}
$I$ is increasing on $(-\infty, t_*]$ and decreasing on $[t_*, +\infty)$.
\end{enumerate}
\end{lemma}

\begin{proof}
Fix a regular $K \in \cC^2_\pi$.
Consider the one-parameter subgroup $L_t \coloneqq \left( \begin{smallmatrix} e^{-t/2} & 0 \\ 0 & e^{t/2} \end{smallmatrix} \right) $ of $\SL(2,\R)$. Then $L_t(E_s) = E_{t+s}$ and therefore intersection functions of the images of $K$ under the subgroup are identical up to translations:
$$
I_K (t+s) = I_{L_{-t} K}(s) \, .
$$
Also note that regularity is invariant under the action of the subgroup.
The regularity property (\ref{i.reg2}) guarantees that $K$ fulfills the hypothesis of \cref{p.tangency} and therefore the function $I_K$ is $C^1$ on a neighborhood of $0$.
By invariance, $I_K$ is $C^1$ on the whole line, which is statement (\ref{i.collect1}) of the \lcnamecref{l.collect}.
Similarly, bearing in mind regularity property (\ref{i.reg2}), we see that \cref{p.formulas} implies that $I_K$ is $C^2$ on the set $\R \setminus T$, which is statement (\ref{i.collect2}), and that \cref{p.key} implies that the derivatives of $I_K$ satisfy the bounds stated in (\ref{i.collect3}).

The function $I_K$ obeys the inequalities $0 < I_K < \min\{\pi,\area{K} \}$ on the whole line;
the second inequality is a consequence of the assumption that $K \in \cC^2_\pi$.
Since the function $I_K$ vanishes at $\pm \infty$, it has critical points. 
Let $t_*$ be one of these.
On the one hand, if $t_* \not\in T$ then  $I_K$ is actually $C^2$ on a neighborhood $V$ of $t_*$ and $I_K''(t_*) < 0$ there. So, reducing the neighborhood $V$ if necessary, the function $I_K'$ becomes decreasing on $V$.
On the other hand, if $t_* \in T$ then 
we can find a neighborhood $V$ of $t_*$ such that $I_K$ is $C^2$ on $V \setminus \{t_*\}$ and $I_K'' < 0$ there. So the function $I_K'$ is decreasing on $V \setminus \{t_*\}$, and since it is continuous, it is actually decreasing on $V$.

We have shown that every critical point of $I_K$ is isolated and is a local maximum. Therefore the critical point, which exists, is unique, proving statement~(\ref{i.collect4}).
We have also seen that the critical point is a local maximum, and so statement~(\ref{i.collect5}) follows.
\end{proof}

\begin{proof}[Proof of \cref{p.new_qqc}]
Let $\epsilon>0$ and $r>0$ be given.
Without loss of generality, we assume $r\le 1$.
Let $\delta = \delta(\epsilon) > 0$ be given by \cref{p.key} and let
$\eta \coloneqq \min \bigl\{ \epsilon, \tfrac{1}{2}\delta r^2 \bigr\}$.
Fix a regular $K \in \cC^2_\pi$ and for simplicity write $I \coloneqq I_K$ and $b \coloneqq \min\{\pi, \area{K}\}$.
Fix the three numbers $t_0<t_1<t_2$ satisfying the assumptions, namely
$t_1 \in [t_0+r, t_2-r]$ and $I(t_1) \in [2\epsilon, b-2\epsilon]$.
We suppose that $t_1 \ge t_*$, where $t_*$ is the critical point of $I$, the other case being analogous.
Since $I(t_2) \le I(t_1 + r)$, it is sufficient to prove that:
\begin{equation*} 
I(t_1 + r) \le I(t_1) - \eta \, .
\end{equation*}
This clearly holds if $I(t_1 + r) < \epsilon$, so assume that $I(t_1 + r) \ge \epsilon$.

Next, suppose $I' \le -\delta$ over the interval $J \coloneqq [t_1,t_1+r]$.
Then, by the Mean Value Theorem, $I(t_1 + r) - I(t_1) \le -\delta r < - \eta$, 
completing the proof in this case.
So assume that $I' > -\delta$ somewhere on $J$.

Note that $I' \le 0$ and $\epsilon \le I \le b-2\epsilon$ on the interval  $J$.
It follows from part~\ref{i.collect3} of \cref{l.collect} that $\max \{-I'(t),-I''(t)\} > \delta$ for every $t \in J$ except a finite number of points where the second derivative may not be defined.
So $I'$ is decreasing on the set $S \coloneq \{ t \in J \st I'(t) > -\delta\}$, which is nonempty by assumption. It follows that $S$ must be an interval with left endpoint $t_1$. Let $s$ be the right endpoint. Then:
$$
t \in [t_1,s] \ \Rightarrow \ 
I'(t)  = \underbrace{I'(t_1)}_{\le 0} + \int_{t_1}^{t} \underbrace{\vphantom{I'(t_1)} I''(u)}_{<-\delta} \dd u \le -\delta (t-t_1) \, ,
$$
while
$$
t \in [s,t_1+r] \ \Rightarrow \ 
I'(t) \le -\delta \le -\delta (t-t_1) \text{ as well}
$$
(using that $r \le 1$).
So:
\begin{equation*}
I(t_1 + r) - I(t_1) 
\le \int_{t_1}^{t_1+r} -\delta (t-t_1) \dd t 
=   -\tfrac{1}{2} \delta r^2 
\le - \eta \, ,
\end{equation*}
as we wanted to show.
\end{proof}

As explained in \cref{ss.setup}, \cref{t.qc} follows.

\subsection{An extension of Theorem~\ref{t.qc}}\label{ss.extension}

The intersection function $I_K$ of any $K \in \cC^2$ is always bounded by the value $\min\{\pi,\area{K} \}$. 
If $K \not\in \cC^2_\pi$ then $I_K$ may have a plateau at this value and therefore may fail to be strictly quasiconcave.
Therefore the assumption $K \in \cC^2_\pi$ cannot be removed altogether from \cref{t.qc}.
On the other hand, this assumption is only used to guarantee that $I_K < \min\{\pi,\area{K} \}$ everywhere.
In fact, it is straightforward to modify the proof of \cref{t.qc} and obtain the following result:

\begin{theorem}\label{t.extension}
Let $K \in \cC^2$.
Let $J \subseteq \R$ be an interval such that $I_K (t) < \max\{\pi, \area{K}\}$ for every $t \in J$.
Then the restriction of $I_K$ to $J$ is a strictly quasiconcave function.
\end{theorem}

\section{Discarding ellipses with displaced centers: Proof of Lemma~\ref{l.discard}} \label{s.discard}

Let us finally prove \cref{l.discard}, which, as seen in the introduction, allows us to deduce the uniqueness \cref{t.unique} from \cref{t.qc}.
We rely on the following result, which is essentially a corollary of the Brunn--Minkowski inequality and holds in arbitrary dimension:

\begin{proposition}[Zalgaller~\cite{Z}]\label{p.Z}
Let $K_1$, $K_2 \subset \R^d$ be convex bodies.
Let $M$ be the set of $v \in \R^d$ that maximize of the volume of $K_1 \cap (K_2+v)$.
Then $M$ is a nonempty compact convex set, and the sets $K_1 \cap (K_2+v)$ with $v \in M$ are identical up to translation.
\end{proposition}

\begin{proof}[Proof of \cref{l.discard}]
Let $K \subset \R^2$ be a centrally symmetric convex body.
For a contradiction, suppose that $K$ admits an MI ellipse with area $\lambda$ in the range $\area{\fJ_K} <  \lambda < \area{\fL_K}$ which is not centrally symmetric, and write it as $E + v_0$, where $E$ is centrally symmetric and $v_0 \neq 0$.
Applying an appropriate linear map if necessary, we can assume that $E$ is the unit disk $E_0$ and that $v_0$ is horizontal, i.e.\ $v_0 = (\epsilon_0, 0)$.

Let $M \ni v_0$ be the set of $v \in \R^2$ such that that maximize $\area{K \cap (E_0 + v)} = \area{(K-v)\cap E_0}$, which by \cref{p.Z} is compact and convex. Since $K$ and $E_0$ are centrally symmetric, so is $M$.
In particular, $0 \in M$ and $E_0$ is also an MI ellipse for $K$.

The \lcnamecref{p.Z} also says that for each $v \in M$ the set $(K+v) \cap E_0$ is a translate of $K \cap E_0$, say $(K\cap E_0) + u$. Consider some $z \in K \cap \partial E_0$ (which exists since $K \not \subseteq E_0$). Then both points $z + u$ and $-z + u$ belong to $E_0$, which forces $u = 0$. We have shown that the sets $(K+v) \cap E_0$ with $v \in M$ are actually identical: no translation is needed.

Since $M$ contains the segment $[-v_0,v_0]$, for every $z \in K \cap E_0$, the intersection of the segment $z+[-v_0,v_0]$ with $E_0$ is contained in $K$.
By overlapping such segments, we conclude that the intersection of the line $z + \R v_0$ with $E_0$ is contained in $K$. This property implies that $K \cap E_0$ equals $S \cap E_0$, where $S = \R \times [-b,b]$ is a strip in the plane.
Since $E_0 \not \subseteq K$, we must have $0<b<1$.
Using that $K \cap (E_0 \pm v_0) = (K \cap E_0) \pm v_0$, we conclude that there is a neighborhood $V$ of the unit disk $E_0$ such that $V \cap K = V \cap S$.
In particular, $\partial K$ is transverse to the unit circle $\partial E_0$ and there are $4$ crossings, namely $(\pm \sqrt{1-b^2}, \pm b)$.
Therefore we may apply \cref{p.formulas}, and conclude that if $\xi_1 \coloneqq \arcsin b$ then 
$I'_K(0) = \sin 2 \xi_1 > 0$.
So for sufficiently small $t>0$, the set $E_{-t} \cap K$ has a bigger area than $E_0 \cap K$, which contradicts the fact that $E_0$ is an MI ellipse for $K$.
\end{proof}

\section{Analysis of maximum intersection positions}\label{s.MI_position}

\subsection{Characterization of MI positions for the transverse case}

Recall from \cref{s.derivatives} that two Jordan curves in the plane are called transverse if each of them is of class $C^1$ at a neighborhood of each point of intersection, and that these intersections are transverse in the usual sense.

\begin{theorem}\label{t.sum_squares}
Let $K \subset \R^2$ is a compact convex centrally symmetric set whose boundary $\partial K$ is transverse to the unit circle $S^1$. 
Let $\zeta_1$, \dots, $\zeta_{4n}$ be the points of intersection, cyclically ordered.
Then $K$ is in MI position if and only if
\begin{equation}\label{e.two_sums}
\sum_{j \ \mathrm{odd}} \zeta_j^2 =  \sum_{j \ \mathrm{even}} \zeta_j^2 \, ,
\end{equation}
where we identify $\R^2$ and $\C$ in the usual way.
\end{theorem}

\begin{proof}
Write $\zeta_j = e^{\mathrm{i} \xi_j}$. 
If $K$ is in MI position then the derivative given by formula \eqref{e.1st} vanishes; moreover, the same is true if we apply a rotation to $K$, i.e., replace each $\xi_j$ by $\xi_j + \phi$.
Therefore:
$$
\sum_{j=1}^{2n} (-1)^j \sin 2(\xi_j+\phi) = 0 \quad \text{for all } \phi \in \R \, . 
$$
Using a trigonometric identity, we see that the latter condition is equivalent to:
$$
\sum_{j=1}^{2n} (-1)^j \cos 2 \xi_j  = \sum_{j=1}^{2n} (-1)^j \sin 2 \xi_j  = 0 \ ,
$$
which is condition \eqref{e.two_sums}.

Conversely, suppose that condition \eqref{e.two_sums} holds.
Then, reversing the arguments above, we obtain that for every $\phi \in \R$, the intersection function of the rotated convex body $e^{\mathrm{i} \phi} K$ is $C^1$ and its derivative at $t = 0$ vanishes.
Furthermore, by \cref{p.key}, the second derivative is also defined and is negative.
Hence, among centrally symmetric ellipses of area $\pi$, the unit disk $E_0$ attains a local maximum for the area of intersection with $K$. 
If we knew that $K \in \cC^2_\pi$ (and therefore $e^{\mathrm{i} \phi} K \in \cC^2_\pi$ for every $\phi$) then we could apply the strict quasiconcavity \cref{t.qc} and conclude that this local maximum is the global maximum, that is, $K$ is in MI position.
In order to conclude the proof we will show that $K \in \cC^2_\pi$, that is, $\area{\fJ_K} < \pi < \area{\fL_K}$.

Suppose for a contradiction that the John ellipse $\fJ_K$ has area $\ge \pi$.
Since $\fJ_K$ is centrally symmetric, it follows that $K$ contains a centrally symmetric ellipse $E$ of area $\pi$.
This ellipse $E$ cannot be the disk $E_0$ since we are assuming that $K$ and $E_0$ have transverse boundaries.
Applying a rotation if necessary, we can assume that $E = E_{t_*}$ for some $t_* > 0$.
Then the intersection function $I_K$ satisfies $I_K(t_*) = \pi$.
Reducing $t_*$ if necessary, we can assume that $I_K(t) < \pi$ for all $t$ in the interval $J \coloneqq [0, t_*)$.
As seen before, the function $I_K$ attains a local maximum at $0$.
It follows the function $I_K$ attains a local minimum somewhere in the interior of $J$.
This contradicts \cref{t.extension}.
Therefore $\area{\fJ_K} \le \pi$. 
A similar reasoning proves that $\area{\fL_K} \ge \pi$. 
So $K \in \cC^2_\pi$, as claimed, and the \lcnamecref{t.sum_squares} follows.
\end{proof}

\cref{p.48} is actually a corollary:

\begin{proof}[Proof of \cref{p.48}]
Suppose that $K$ is MI position with boundary transverse to the unit circle and intersecting it at the points $\zeta_1$, \dots, $\zeta_{4n}$, listed in counterclockwise order. Note that $\zeta_{j+2n}= - \zeta_j$. 

If $n=1$ then by \eqref{e.two_sums} we would have $\zeta_1^2 = \zeta_2^2$, i.e., $\zeta_1$ and $\zeta_2$ are antipodal; absurd. This proves part (\ref{i.4}).

Now suppose $n=2$.
We want to prove that $\zeta_{j+2} = \mathrm{i} \zeta_j$.
Condition \eqref{e.two_sums} becomes:
$$
\zeta_1^2 + \zeta_3^2 = \zeta_2^2 + \zeta_4^2 \, . 
$$
So we must prove that both sides of this equation vanish.
Suppose that is not the case. 
Observe that a pair of non-antipodal points in the unit circle is uniquely determined (modulo permutation) by their midpoint.
Therefore $\{\zeta_1^2 , \zeta_3^2 \} = \{\zeta_2^2 , \zeta_4^2 \}$.
But $\zeta_1^2$, $\zeta_2^2$, $\zeta_3^2$, $\zeta_4^2$ are distinct (and cyclically ordered).
We have reached a contradiction.
This proves part (\ref{i.8}).
\end{proof}

\subsection{Comparison with the classical characterization of John position}

Let $\mu$ be a positive (and nonzero) Borel measure on the unit sphere $S^{d-1}$.
We say that $\mu$ is \emph{balanced} if
$$
\int_{S^{d-1}} u \dd\mu(u) = 0 \, ,
$$
that is, the center of mass of $\mu$ is the origin.
We say that $\mu$ is \emph{isotropic} if, for some $c>0$,
$$
\forall v \in \R^d \, , \quad 
\int_{S^{d-1}} \langle u,v \rangle^2 \dd\mu(u) = c \|v\|^2 
$$
(where $\| \mathord{\cdot} \|$ denotes euclidian norm),
that is, the inertia ellipsoid of $\mu$ with respect to the origin is round.
One necessarily has $c = \frac{1}{d}\mu(S^{d-1})$.
See e.g.\ \cite[\S~10.13]{Schneider} for several uses of isotropic measures in convex geometry.

We say that a convex body in $K\subset \R^d$ is in \emph{John position} if its John ellipsoid is the euclidian unit ball. The following theorem is well-known:

\begin{theorem}[John] \label{t.John}
If a convex body $K\subset \R^d$ is in John position then there exists a balanced isotropic measure supported on $S^{d-1} \cap \partial K$.  
\end{theorem}

Here we prove a similar result for the planar MI position:

\begin{proposition}\label{p.isotropic}
If a centrally symmetric body $K \subset \R^2$ is in MI position then there exists a balanced isotropic measure supported on $S^1 \cap \partial K$.  
\end{proposition}

See the paper \cite{AAK} for another result on the existence of balanced isotropic measures for bodies in MI position, under certain generic assumptions, and without restriction on dimension.

\begin{proof}[Proof of \cref{p.isotropic}]
It suffices to consider the case of $\partial K$ transverse to $S^1$; the general case then follows by perturbation and using the fact that balanced isotropic measures form a weakly-$*$-closed set.

Let $\zeta_1$, \dots, $\zeta_{4n}$ be the points in $S^1 \cap \partial K$, cyclically ordered.
Note that $\zeta_{j+2n}= - \zeta_j$. 

We claim that the convex hull of the points $\zeta_j^2$ contains the origin.
If not, there exists a line $L \subset \R^2$ through the origin such that all points $\zeta_j^2$ belong to the same connected component of $\R^2 \subset L$.
Let $P \colon \R^2 \to L$ be the orthogonal projection onto $L$.
Since the points $\zeta_1^2$, \dots, $\zeta_{2n}^2 \in S^1$ are distinct,
their projections $z_1 \coloneqq P(\zeta_1^2)$, \dots, $z_{2n} \coloneqq P(\zeta_{2n}^2)$ are distinct. 
Color each $z_j$ red or blue depending on whether $j$ is odd or even.
Since the points $\zeta_1^2$, \dots, $\zeta_{2n}^2 \in S^1$ are cyclically ordered, the colors of the points $z_1$, \dots, $z_{2n} \in L$ alternate.
In particular, equation \eqref{e.two_sums} cannot hold, since the two sums have different projections. 
By \cref{t.sum_squares}, the set $K$ is not in MI position. 
This contradicts the assumption, and therefore we proved that the convex hull of the points $\zeta_j^2$ contains the origin.

Hence there exist weights $p_1$, \dots, $p_{4n} \ge 0$ such that $\sum_j p_j = 1$ and $\sum_j p_j \zeta_j^2 = 0$.
Since $\zeta_{j+2n}^2 = \zeta_j^2$, we can assume that the weights satisfy $p_{j+2n}= p_j$ (indices taken mod~$4n$).  
Then the measure on $S^1 \cap \partial K$ defined by $\mu \coloneqq \sum_j p_j \delta_{\zeta_j}$ is 
balanced. 
Let us check that $\mu$ is isotropic.
Note that the (real) euclidian inner product in $\C$ is given by the formula $\langle u,v \rangle = \Re(u \bar{v})$ and so satisfies the identity:
$$
\langle u,v \rangle^2 = \frac{|u^2 v^2| + \langle u^2, v^2 \rangle}{2}  \, .
$$ 
Using this, we calculate, for arbitrary $v \in \C$,
$$
\int_{S^1} \langle u, v \rangle^2 \dd\mu(u) =
\sum_j p_j \langle \zeta_j, v \rangle^2 =
\frac{\sum_j p_j}{2} |v|^2  + \frac{1}{2} \left\langle {\textstyle\sum_j p_j \zeta_j^2} , v^2 \right\rangle = \frac{|v|^2}{2}  \, ,
$$
proving that $\mu$ is isotropic.
\end{proof}

\begin{remark}
The converse of \cref{t.John} also holds, as shown by Ball \cite{Ball92}.
However, \emph{the converse of \cref{p.isotropic} is false.}
Indeed, fix any $\epsilon \neq \frac{\pi}{6}$ in the range $0 < \epsilon < \frac{\pi}{4}$, and let $Z$ be the set consisting of $12$ points in the unit circle that is invariant by a quarter turn and contains $\{e^{-\mathrm{i}\epsilon}, 1, e^{\mathrm{i}\epsilon}\}$.
Then the equidistributed probability measure on $Z$ is balanced and isotropic.
Take a centrally symmetric convex body $K \subset \R^2$ whose boundary $\partial K$ is transverse to the unit circle $S^1$ and intersects it exactly on $Z$. 
Then condition \eqref{e.two_sums} does not hold; indeed the two sums are 
$2 - 4 \cos 2\epsilon \neq -2 + 4 \cos 2\epsilon$.
By \cref{t.sum_squares}, the set $K$ is not in MI position. 
\end{remark}

\section{Directions for future research}\label{s.questions}

We pose a few questions:

\begin{enumerate}[label={\thesection.\arabic*.},ref={\thesection.\arabic*}]

\item\label{i.q_asymmetric}
Can the assumption of central symmetry be removed from the main \cref{t.unique}?

\item\label{i.q_unconstrained_area} 
Given a (say, centrally symmetric) convex body $K \subset \R^2$, is there a unique ellipse that best approximates it with respect to symmetric difference metric (without constraining its area)?

\item\label{i.q_normalized_SD} 
The previous question for the normalized symmetric difference metric $d_\mathrm{nsym}$, defined by \eqref{e.metrics}.

\item\label{i.q_log_concave} 
Given an arbitrary $K \in \cC^2_\pi$, is the intersection function $I_K$ log-concave? (See \cite[\S~4]{AAK} for a stronger conjecture, motivation, and relations with known results.)

\end{enumerate}

Some of these questions are possibly accessible with the methods of this paper.
In any case, the investigation of the higher-dimensional versions of \cref{t.unique,t.qc} and of the questions above will require new methods.

\medskip

\begin{ack}
I thank Paula Porto for drawing most of the figures.
I thank W{\l}odek Kuperberg for posing the problem that motivated this paper, for telling me that ellipses with displaced centers could be discarded, for pointing to reference \cite{Z}, and for suggesting Question~\ref{i.q_normalized_SD}.
I am grateful to the referees for corrections, references, and criticism that allowed me to improve the paper significantly. I particularly thank one of the referees for suggesting to go beyond \cref{q.Wlodek} and to consider the full family of ellipses interpolating between John and Loewner ellipses, and also for posing questions that led to the results presented in \cref{s.MI_position}.
\end{ack}


\bigskip

\footnotesize
\textit{Web-page:} \href{http://www.mat.uc.cl/~jairo.bochi/}{\tt www.mat.uc.cl/\textasciitilde jairo.bochi}
\medskip

\end{document}